\documentclass{article}
\usepackage[utf8]{inputenc}
\usepackage{amsmath,amssymb,amsthm,tikz,url,multirow, enumerate}

\newtheorem{theorem}{Theorem}
\newtheorem{lemma}{Lemma}

\title{Pattern Avoidance in Parking Functions}
\author{Ayomikun Adeniran\\
\small Department of Mathematics\\[-0.8ex]
\small Colby College\\[-0.8ex] 
\small Maine, U.S.A.\\
\small\tt aadenira@colby.edu\thanks{The first author is supported in part by an AMS-Simons Travel Grant.}
\and
Lara Pudwell\\
\small Department of Mathematics and Statistics\\[-0.8ex]
\small Valparaiso University\\[-0.8ex]
\small Indiana, U.S.A\\
\small\tt Lara.Pudwell@valpo.edu}
\date{\today}

\begin{document}

\maketitle

\begin{abstract}
In this paper, we view parking functions viewed as labeled Dyck paths in order to study a notion of pattern avoidance first considered by Remmel and Qiu.  In particular we enumerate the parking functions avoiding any set of two or more patterns of length 3, and we obtain a number of well-known combinatorial sequences as a result. Along the way, we find bijections between specific sets of pattern-avoiding parking functions and a number of combinatorial objects such as partitions of polygons and trees with certain restrictions.
\end{abstract}

\section{Introduction}

Let $\mathcal{S}_n$ be the set of all permutations on $[n]=\{1,2,\dots, n\}$. Given $\pi \in \mathcal{S}_n$ and $\rho \in \mathcal{S}_m$ we say that $\pi$ \emph{contains} $\rho$ as a pattern if there exist $1 \leq i_1 < i_2 < \cdots < i_m \leq n$ such that $\pi_{i_a} \leq \pi_{i_b}$ if and only if $\rho_a \leq \rho_b$.    In this case we say that $\pi_{i_1}\cdots \pi_{i_m}$ is \emph{order-isomorphic} to $\rho$, and that $\pi_{i_1} \cdots \pi_{i_m}$ is an \emph{occurrence} of $\rho$ in $\pi$.  If $\pi$ does not contain $\rho$, then we say that $\pi$ \emph{avoids} $\rho$.  For example $\pi=364521$ contains the pattern $\rho = 2341$ because the digits of $\pi_1\pi_3\pi_4\pi_5=3452$ have the same relative order as the digits of $\rho$; this is one of two instances of 2341 in $\pi$.  Simion and Schmidt \cite{SS85} enumerated permutations avoiding a pair of patterns of length 3 and many further enumerative results followed (see B\'{o}na \cite{B12} and Kitaev \cite{K11} for surveys) both for patterns in permutations and in words.  We now consider this notion of pattern applied to the area of parking functions.   

Suppose that $n$ cars need to park in $n$ parking spaces along a one-way street.  Each car has a favorite parking spot.  Each car will traverse the street exactly once and it will take its favorite spot or the next available spot after that.  A collection of preferences that results in all the cars being parked is called a parking function.  As extreme examples: if all cars prefer spot 1, then we have a parking function since the $i$th car will end up in spot $i$; however if all cars prefer a different spot, such as spot 2 or spot $n$, we do not have a parking function since all the cars will bypass spot 1 and there are not enough remaining spots to park the cars.  

It is well known that a collection of preferences is a parking function if and only if for all $1 \leq i \leq n$ at least $i$ cars prefer spot $i$ or earlier.  Further, it is known (see \cite{KW66, P59}) that there are $(n+1)^{n-1}$ possible parking functions for $n$ cars. For a comprehensive survey of results on parking functions, see $\cite{Y15}$.

There are multiple ways to represent parking functions that associate them with words and permutations.  For example, Jel\'{i}nek and Mansour \cite{JM09} view a parking function as a word $w \in [n]^n$ where for every $i= 1, \dots, n$, $w$ has at least $i$ letters less than or equal to $i$. In this representation, the three parking functions on 2 letters are represented as 11, 12, and 21. They apply the standard definition of pattern avoidance to these words and determine pairs of patterns $\pi$ and $\pi^\prime$ where the number of $\pi$-avoiding parking functions of size $n$ equals the number of $\pi^\prime$-avoiding parking functions of size $n$ for all $n$.

In a different direction, Garsia and Haiman  \cite{GH96, H08} represent parking functions as Dyck paths with labeled up-steps.  Here a Dyck path is a lattice path from $(0,0)$ to $(n,n)$ in the Cartesian plane consisting of $n$ north-steps (N) and $n$ east-steps (E) that never crosses below the line $y=x$.  Dyck paths are counted by the Catalan numbers, and there is a natural bijection with legally-arranged sequences of parentheses (by changing each N to a left parenthesis and each E to a right parenthesis).  In such a path, we may label each of the $n$ north-steps with a distinct integer from $\{1,\dots,n\}$ such that consecutive north-steps must have their labels in increasing order.  In this representation, the labels of north-steps along $x=i$ correspond to the cars who prefer spot $i+1$.  The three parking functions corresponding to 11, 12, and 21 in Jel\'{i}nek and Mansour's notation are shown as labeled Dyck paths in Figure \ref{fig:smallpf}.

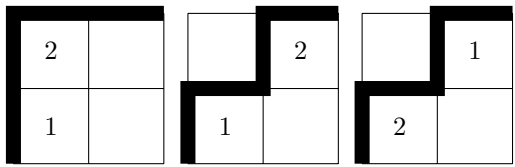
\begin{figure}
    \centering
\begin{tikzpicture}
\draw[step=1.0,black,thin,xshift=0.5cm,yshift=0.5cm] (0,0) grid (2,2);
\draw[line width=0.2cm] (0.5,0.5) -- (0.5,2.5)--(2.5,2.5);
\node at (1,1) {1};
\node at (1,2) {2};
\end{tikzpicture}
\begin{tikzpicture}
\draw[step=1.0,black,thin,xshift=0.5cm,yshift=0.5cm] (0,0) grid (2,2);
\draw[line width=0.2cm] (0.5,0.5) -- (0.5,1.5)--(1.5,1.5)--(1.5,2.5)--(2.5,2.5);
\node at (1,1) {1};
\node at (2,2) {2};
\end{tikzpicture}
\begin{tikzpicture}
\draw[step=1.0,black,thin,xshift=0.5cm,yshift=0.5cm] (0,0) grid (2,2);
\draw[line width=0.2cm] (0.5,0.5) -- (0.5,1.5)--(1.5,1.5)--(1.5,2.5)--(2.5,2.5);
\node at (1,1) {2};
\node at (2,2) {1};
\end{tikzpicture}
    \caption{The three parking functions of size 2}
    \label{fig:smallpf}
\end{figure}

In Garsia and Haiman's notation, each parking function is represented uniquely by the combination of a Dyck path and the labels on its north-steps.  However, there is also an interesting many-to-one correspondence between permutations and parking functions that manifests naturally as a result of this representation.  In particular, given a parking function represented as a labeled Dyck path, reading the labels of the north-steps in order produces a unique permutation.  Multiple parking functions may correspond to the same permutation; for example, the first two parking functions in Figure \ref{fig:smallpf} correspond to the permutation 12, while the last one corresponds to the permutation 21.  However, this many-to-one correspondence gives another way to study pattern avoidance in parking functions:  consider a parking function $p$ represented as a labeled Dyck path.  Let $\phi(p)$ be the permutation obtained by reading the labels of $p$'s north steps in order.  Further, let $\mathrm{pf}_n(\rho)$ be the number of parking functions $p$ of size $n$ such that $\phi(p)$ avoids $\rho$, and by extension let $\mathrm{pf}_n(\rho_1,\rho_2,\dots, \rho_m)$ be the number of parking functions of size $n$ such that $\phi(p)$ avoids each of $\rho_1, \dots, \rho_m$ simultaneously.

It is quick to see that $$\mathrm{pf}_n(12)=1.$$  The only way to avoid the pattern 12 is to consider the Dyck path of the form $(NE)^n$ whose labels appear in decreasing order, as in the last image of Figure \ref{fig:smallpf}.  It is similarly quick to see that $$\mathrm{pf}_n(21)=C_n,$$ where $C_n=\frac{\binom{2n}{n}}{n+1}$ is the $n$th Catalan number.  Here, for any of the $C_n$ possible Dyck paths, the north-steps must be labeled in increasing order, as in the first two images of Figure \ref{fig:smallpf}.

Remmel and Qiu \cite{RQ17} considered the case of parking functions avoiding 123 and determined that
$$\mathrm{pf}_n(123)=\sum_{k=n/2}^{n}\frac{C_k}{n-k+1}\binom{n}{k}\binom{k}{n-k},$$
where $C_k$ is the $k$th Catalan number, but they did not conduct a more systematic exploration.

While Garsia and Haiman's representation can be visualized in terms of Dyck paths, for ease of notation, we may also write parking functions as ordered collections of sets, where the members of the $i$th set are the labels of the north-steps along $x=i-1$ in the Dyck representation. We separate sets by vertical lines.  In this \emph{block notation}, the three Dyck paths of Figure \ref{fig:smallpf} are written as $\{1,2\}\vert \emptyset$, $\{1\}\vert \{2\}$, and $\{2\}\vert \{1\}$ respectively.

In this paper, we continue Remmel and Qiu's study of pattern avoidance using Garcia and Haiman's labeled Dyck path representation of parking functions.  In particular, we study parking functions that avoid any collection of permutation patterns of length 3. We achieve a number of famous combinatorial sequences as well as a number of new sequences. The rest of this paper is organized according to the number of patterns to be avoided.  By the Erd\H{o}s-Szekeres Theorem, any permutation that avoids both 123 and 321 must have length 4 or less, so we only consider parking functions that avoid at most one of these patterns at a time.  While some results follow from case-work, other arguments require new bijections with combinatorial objects or computer-assisted induction.

Since the definition of pattern considered in this paper has its roots in permutation patterns, we will find it helpful to make use of certain permutation notation.  Throughout this paper, $\mathcal{S}_n$ denotes the set of all permutations of length $n$ and $\mathcal{S}_n(P)$ denotes the set of all permutations of length $n$ that avoid all patterns in the list $P$.  $I_n = 12\cdots n$ denotes the increasing permutation of length $n$ and $J_n=n\cdots 21$ denotes the decreasing permutation of length $n$.  Further, given permutations $\alpha \in \mathcal{S}_n$ and $\beta \in \mathcal{S}_m$, then $\alpha \oplus \beta$ and $\alpha \ominus \beta$ are permutations in $\mathcal{S}_{n+m}$ known respectively as the \emph{direct sum} and the \emph{skew sum} of $\alpha$ and $\beta$.  In particular, $$(\alpha \oplus \beta)_i=\begin{cases}\alpha_i&1 \leq i \leq n\\
\beta_{i-n}+n&n+1\leq i \leq n+m
\end{cases}$$ and  $$(\alpha \ominus \beta)_i=\begin{cases}\alpha_i+m&1 \leq i \leq n\\
\beta_{i-n}&n+1\leq i \leq n+m
\end{cases}.$$

\section{Avoiding five patterns}

Since we omit sets that avoid 123 and 321 simultaneously, there are only two sets of size 5 to consider.

\begin{theorem}
$$\mathrm{pf}_n(123, 132, 213, 231,312)=\begin{cases}3&n=2,\\
1&\text{otherwise.}
\end{cases}$$
\end{theorem}

\begin{proof}
For $n \geq 3$, all subsequences of length 3 must form 321 patterns.  In other words, the labels of the north-steps must all be in decreasing order.  This implies that each north-step is on a different vertical line in the Cartesian plane, so the only relevant parking function is the one of the form $(NE)^n$ with decreasing labels.
\end{proof}

\begin{theorem}
$$\mathrm{pf}_n(132, 213, 231, 312, 321)=\begin{cases}3&n=2,\\
C_n&\text{otherwise.}
\end{cases}$$
\end{theorem}

\begin{proof}
For $n \geq 3$, all subsequences of length 3 must form 123 patterns.  In other words, the labels of the north-steps must all be in increasing order.  Each of the $C_n$ possible Dyck paths can be labeled once in this way.
\end{proof}

\section{Avoiding four patterns}

There are $\binom{6}{4}=15$ distinct sets of four patterns of length 3. However, 6 of these sets involve both 123 and 321 and are addressed by the Erd\H{o}s-Szekeres Theorem.  From the remaining 9 sets of patterns, we get five different enumeration sequences.  A summary of results is shown in Table \ref{tab:four_patterns}.

\begin{table}[hbt]
    \centering
    \begin{tabular}{|c|c|c|c|}
    \hline
    Patterns $P$&$\mathrm{pf}_n(P)$, $1 \leq n \leq 6$&OEIS&Result\\
\hline
123, 132, 213, 231 &\multirow{2}{*}{1, 3, 3, 3, 3, 3}&\multirow{2}{*}{A122553}&\multirow{2}{*}{Theorem \ref{T:fourd}}\\
123, 132, 231, 312 &&&\\
\hline

123, 132, 213, 312 &\multirow{2}{*}{1, 3, 4, 5, 6, 7}&\multirow{2}{*}{A065475}&\multirow{2}{*}{Theorem \ref{T:foure}}\\
123, 213, 231, 312 &&&\\
    \hline
132, 213, 231, 312&1, 3, 6, 15, 43, 133&new&Theorem \ref{T:foura}\\
\hline
132, 213, 231, 321&\multirow{2}{*}{1, 3, 7, 19, 56, 174}&\multirow{2}{*}{A071716}&\multirow{2}{*}{Theorem \ref{T:fourb}}\\
132, 231, 312, 321&&&\\
\hline
132, 213, 312, 321&\multirow{2}{*}{1, 3, 8, 23, 70, 222}&\multirow{2}{*}{A000782}&\multirow{2}{*}{Theorem \ref{T:fourc}}\\
213, 231, 312, 321&&&\\
\hline
    \end{tabular}
    \caption{\textit{Enumeration data for parking functions avoiding four patterns of length 3}}
    \label{tab:four_patterns}
\end{table}

\begin{theorem}
\label{T:fourd} For $n \geq 2$,
$$\mathrm{pf}_n(123, 132, 213, 231)=\mathrm{pf}_n(123, 132, 231, 312)=3.$$
\end{theorem}

\begin{proof}
First we claim that any permutation $\pi$ that avoids 123, 132, 213, and 231 must be of the form $J_n$ or $J_{n-2} \ominus I_2$. To see this, notice that $\pi_1=n$ and inductively, since any subpermutation $\tau$ of $\pi$ avoids the same set of patterns, the next entry $\pi_2=n-1$. Continuing this way, it is quick to see that $\pi_k=n-k+1$ for each $1\leq k\leq n-3$. The only possibilities for the last three entries of $\pi$ are the only two patterns of length 3 not being avoided, that is, 312 and 321. Next, it is easy to see that there are exactly two parking functions associated with $\pi=J_{n-2} \ominus I_2$, namely, $\{n\}|\{n-1\}|\{n-2\}| \cdots |\{3\}|\{1\}|\{2\}$ and $\{n\}|\{n-1\}|\{n-2\}| \cdots |\{3\}|\{1,2\}|\emptyset$. Together with the unique parking function corresponding to $\pi=J_n$, we obtain the desired result.

For $\pi$ in $\mathcal{S}_n(123,132,231,312)$, either $\pi_1=n$ or $\pi_n=n$, since otherwise $\pi$ would contain 132 or 231. In the first case, since $\pi$ avoids 312, $\pi$ must be of the form $J_n$. In the second case, since $\pi$ avoids 123 and 132, $\pi_{n-1}=1$.  Now, since $\pi$ avoids 231, we conclude that $\pi$ has the form $J_{n-1} \oplus 1$.  There are exactly two parking functions associated with $\pi=J_{n-1} \oplus 1$, namely, $\{n-1\}|\{n-2\}| \cdots |\{2\}|\{1\}|\{n\}$ and $\{n-1\}|\{n-2\}| \cdots |\{2\}|\{1,n\}|\emptyset$. Together with the unique parking function corresponding to $\pi=J_n$, we obtain the desired result.
\end{proof}

\begin{theorem}
\label{T:foure} For $n \geq 2$,
$$\mathrm{pf}_n(123,132,213,312)=
\mathrm{pf}_n(123, 213, 231, 312)=n+1.$$
\end{theorem}

\begin{proof}
For any $\pi$ in $\mathcal{S}_n(123,132,213,312)$, $\pi_n=1$. Further, since $\pi$ avoids 123 and 132, then either $\pi_1=n$ or $\pi_1=n-1$. If $\pi_1=n$, then $\pi$ avoiding 312 means that $\pi=J_n$. Otherwise, $\pi_1=n-1$ and since $\pi$ avoids 213 and 312 then $\pi=I_2\ominus J_{n-2}$. In the latter case, there are $n$ possible parking functions, namely, $\{n-1\}|\{n\}|\{n-2\}| \cdots |\{2\}|\{1\}$ or any of the $(n-1)$ parking functions of the form  $\{n-1,n\}|\{n-2\}| \cdots |\{k\}|\emptyset|\{k-1\}|\cdots|\{2\}|\{1\}$. Finally, adding this to the unique parking function corresponding to $\pi=J_n$, we obtain the desired result.

In a similar vein, given $\pi$ in $\mathcal{S}_n(123,213,231,312)$, $\pi_1=1$ or $\pi_1=n$, since $\pi$ avoids 213 and 231. In the second case, since $\pi$ avoids 312, then $\pi$ must be of the form $J_n$. If, on the other hand, $\pi_1=1$, then because $\pi$ avoids 123, $\pi$ must be of the form $1\oplus J_{n-1}$. There are $n$ possible parking functions when $\pi_1=1$, namely, $\{1\}|\{n\}|\{n-1\}|\{n-2\}| \cdots |\{2\}$ or any of the $(n-1)$ parking functions of the form  $\{1,n\}|\{n-1\}|\{n-2\}| \cdots |\{k\}|\emptyset|\{k-1\}|\cdots|\{2\}$. Together with the unique parking function corresponding to $\pi=J_n$, we obtain the desired result.
\end{proof}

\begin{theorem}
\label{T:foura} For $n \geq 2$,
$$\mathrm{pf}_n(132, 213, 231, 312)=C_n+1.$$
\end{theorem}

\begin{proof}
For $n \geq 3$, every subsequence of length 3 needs to form either a 123 pattern or a 321 pattern.  The only permutations that accomplish this are the increasing permutation, which can be associated with any of the $C_n$ possible Dyck paths or the decreasing permutation, which can be associated with the one Dyck path of the form $(NE)^n$.  This yields a total of $C_n+1$ total pattern-avoiding parking functions.
\end{proof}

\begin{theorem}
\label{T:fourb} For $n \geq 2$,
$$\mathrm{pf}_n(132, 213, 231, 321)=\mathrm{pf}_n(132, 231, 312, 321)=C_n+C_{n-1}.$$
\end{theorem}

\begin{proof}
In the case where $\pi \in \mathcal{S}_n(132, 213, 231, 321)$, since $\pi$ avoids 213 and 231, either $\pi_1=1$ or $\pi_1=n$. If $\pi_1=1$, then avoiding 132 means that $\pi=I_n$. On the other hand, if $\pi_1=n$, avoiding 321 means that $\pi$ is of the form $1\ominus I_{n-1}$. The former case yields $C_n$ parking functions. We claim that the latter case yields $C_{n-1}$ parking functions. This is because $n$ is always in a block of size 1 at the beginning of each parking function, hence counting all such parking functions is equivalent to counting all parking functions corresponding to $I_{n-1}$. Hence the claim. Adding up gives a total of $C_n+C_{n-1}$ parking functions.

In the case where $\pi \in \mathcal{S}_n(132,231,312,321)$, $\pi_1 < 3$ (otherwise the digits $\pi_1$, 1, and 2 would form a 312 pattern or a 321 pattern). If $\pi_1=1$, then since $\pi$ avoids 132, $\pi=I_n$.  If $\pi_1=2$, then since $\pi$ avoids 213, $\pi_2=1$, and since $\pi$ avoids 132, $\pi = J_2 \oplus I_{n-2}$.  The case where $\pi=I_n$ corresponds to $C_n$ parking functions while the permutation $J_2\oplus I_{n-2}$ corresponds to $C_{n-1}$ parking functions, since 2 is in a block of size 1 at the beginning of each parking function and so counting all such parking functions is equivalent to counting all parking functions corresponding to $I_{n-1}$.
\end{proof}

\begin{theorem}
\label{T:fourc}For $n \geq 2$,
$$\mathrm{pf}_n(132, 213, 312, 321)=\mathrm{pf}_n(213, 231, 312, 321)=2C_n-C_{n-1}.$$
\end{theorem}

\begin{proof}
When $n=2$, we already know there are $2C_2-C_1=3$ parking functions, so we focus on the case where $n\geq 3$.

Suppose $\pi \in \mathcal{S}_n(213, 231, 312, 321)$.  Then since $\pi$ avoids 312 and 321, either $\pi_{n-1}=n$ or $\pi_n=n$.  In the first case, since $\pi$ avoids 231, we see that $\pi_n={n-1}$, and since $\pi$ avoids 213, $\pi=I_{n-2} \oplus J_2$.  In the second case, since $\pi$ avoids 213, $\pi=I_n$.  There are $C_n$ parking functions corresponding to $I_n$.  For the case where $\pi=I_{n-2} \oplus J_2$, we may count the parking functions corresponding to $I_n$, then exclude the parking functions where $n-1$ and $n$ appear in the same block.  When $n-1$ and $n$ are in different blocks, we trade their locations to obtain a parking function corresponding to $\pi=I_{n-2} \oplus J_2$.  If $n-1$ and $n$ are in the same block, we may treat them as a single digit, so the number of parking functions of $I_n$ where $n-1$ and $n$ are in the same block is given by $C_{n-1}$ and the parking functions corresponding to $I_{n-2} \oplus J_2$ is given by $C_n-C_{n-1}$.  Combining this with the case where $\pi=I_n$ gives $2C_n-C_{n-1}$ parking functions.

Suppose $\pi \in \mathcal{S}_n(132, 213, 312, 321)$.  Then $\pi_1 < 3$; otherwise the digits $\pi_1$, 1, and 2 would form a 312 pattern or a 321 pattern.  If $\pi_1=1$, then since $\pi$ avoids 132, $\pi=I_n$.  If $\pi_1=2$, then $\pi_n=1$; otherwise $\pi_1$, 1, and $\pi_n$ would form a 213 pattern.  Finally, since $\pi_n=1$, we see $\pi = I_{n-1} \ominus 1$, otherwise, $\pi_n=1$ would be involved in a 321 pattern.  There are $C_n$ parking functions corresponding to $I_n$.  To count permutations corresponding to $I_{n-1} \ominus 1$, we count parking functions corresponding to $I_n$ where $n-1$ and $n$ are in different blocks.  If they are in different blocks, we may replace $n$ with 1 and we may replace $k$ with $k+1$ for $1 \leq k \leq n-1$ to obtain a parking function corresponding to $I_{n-1} \ominus 1$.  There are $C_n$ total parking functions corresponding to $I_n$, and $C_{n-1}$ of them have $n-1$ and $n$ in the same block.  Therefore, there are $C_n-C_{n-1}$ parking functions corresponding to $I_{n-1} \ominus 1$.  Combining this with the case where $\pi=I_n$ gives a total of $2C_n-C_{n-1}$ parking functions.

\end{proof}

\section{Avoiding three patterns}

There are $\binom{6}{3}=20$ distinct sets of three patterns of length 3. However 4 of these sets involve both 123 and 321 and are addressed by the Erd\H{o}s-Szekeres Theorem.  From the remaining 16 sets of patterns, we get 10 different enumeration sequences.  A summary of results is shown in Table \ref{tab:three_patterns}.

\begin{table}[hbt]
    \centering
    \begin{tabular}{|c|c|c|c|}
    \hline
    Patterns $P$&$\mathrm{pf}_n(P)$, $1 \leq n \leq 6$&OEIS&Result\\
    \hline
123, 132, 231&1, 3, 5, 7, 9, 11&A005408&Theorem \ref{T:123_132_231}\\
\hline
123, 132, 312&\multirow{3}{*}{1, 3, 6, 10, 15, 21}&\multirow{3}{*}{A000217}&\multirow{3}{*}{Theorem \ref{T:123_132_312}}\\
123, 213, 231&&&\\
123, 231, 312&&&\\
\hline
123, 213, 312&1, 3, 7, 13, 21, 31&A002061&Theorem \ref{T:123_213_312}\\
\hline
123, 132, 213&1, 3, 6, 17, 43, 123&A143363&Theorem \ref{T:123_132_213}\\
\hline
132, 213, 231&\multirow{2}{*}{1, 3, 8, 22, 64, 196}&\multirow{2}{*}{A014138}&\multirow{2}{*}{Theorem \ref{T:132_231_312}}\\
132, 231, 312&&&\\
\hline
132, 213, 312&\multirow{2}{*}{1, 3, 9, 28, 90, 297}&\multirow{2}{*}{A000245}&\multirow{2}{*}{Theorem \ref{T:132_213_312}}\\
213, 231, 312&&&\\
\hline
132, 231, 321&1, 3, 9, 29, 98, 342&A077587&Theorem \ref{T:132_231_321}\\
\hline
132, 213, 321&\multirow{3}{*}{1, 3, 10, 35, 126, 462}&\multirow{3}{*}{A001700}&\multirow{3}{*}{Theorem \ref{T:132_213_321}}\\
132, 312, 321&&&\\
213, 231, 321&&&\\
\hline
213, 312, 321&1, 3, 11, 41, 154, 582&A076540&Theorem \ref{T:213_312_321}\\
\hline
231, 312, 321&1, 3, 10, 38, 154, 654&A001002&Theorem \ref{T:231_312_321}\\
\hline
    \end{tabular}
    \caption{\textit{Enumeration data for parking functions avoiding three patterns of length 3}}
    \label{tab:three_patterns}
\end{table}

\begin{theorem}
\label{T:123_132_231}
$$\mathrm{pf}_n(123,132,231)=2n-1.$$
\end{theorem}

\begin{proof}
Suppose $\pi \in \mathcal{S}_n(123,132,231)$. Because, $\pi$ avoids $123$ and $132$, either $\pi_{n-1}=1$ or $\pi_n=1$. In addition, avoiding $231$ means that no ascent can appear before 1 in $\pi$. Combining these properties, either $\pi=J_n$ or $\pi=J_{n-k} \ominus \left(J_{k-1} \oplus 1\right)$ where $2 \leq k \leq n$. There is one parking function corresponding to the permutation $J_n$. In the second case, since the first $n-1$ digits of $\pi=J_{n-k} \ominus \left(J_{k-1} \oplus 1\right)$ are decreasing they must appear in separate blocks.  We may choose whether to put 1 and $k$ in the same block or in separate blocks, so for each of the $n-1$ choices of $k$, there are two parking functions. Combining both cases, we get a total of $1+2(n-1)=2n-1$ pattern-avoiding parking functions.
\end{proof}

\begin{theorem}
\label{T:123_132_312}
$$\mathrm{pf}_n(123,132,312)=\mathrm{pf}_n(123,213,231)=\mathrm{pf}_n(123,231,312)=\binom{n+1}{2}.$$
\end{theorem}

\begin{proof}
Suppose $\pi \in \mathcal{S}_n(123, 132, 312)$.  Since $\pi$ avoids 132, all digits before $n$ are larger than all digits after $n$.  Since $\pi$ avoids 123, all digits before $n$ are in decreasing order.  Since $\pi$ avoids 312, all digits after $n$ are in decreasing order.  Therefore, $\pi =\left(J_{k} \oplus 1\right) \ominus J_{n-k-1}$, for some $0 \leq k \leq n-1$. There are two possible cases for the structure of such a parking function:
\begin{enumerate}[(i)]
    \item  It has the form $\{n-1\}|\{n-2\}|\cdots |\{n-k\}|\{n\}|\{n-k-1\}| \cdots |\{1\}$.
    \item It has the form $\{n-1\}|\{n-2\}|\cdots |\{n-k, n\}|\{n-k-1\}| \cdots|\emptyset|\cdots  |\{1\}$. 
\end{enumerate}

There are $n$ parking functions of the first form -- one for each choice of $k$.  Given a particular choice of $k \geq 1$, there are $n-k$ parking functions of the second form -- one for each choice of location of the empty block. Hence, the total number of possible parking functions is given by:
$$n+\sum_{k=1}^{n-1} (n-k) = \sum_{k=0}^{n-1} (n-k) = \binom{n+1}{2}.$$

Suppose $\pi \in \mathcal{S}_n(123,213,231)$.  Suppose $\pi_1=1$.  Since $\pi$ avoids 123, we know $\pi=1 \oplus J_{n-1}$.  On the other hand, suppose $\pi_{k+1}=1$ where $k>0$.  Then since $\pi$ avoids 213, all digits before 1 are larger than all digits after 1.  Since $\pi$ avoids 231, all digits before 1 are in decreasing order.  Since $\pi$ avoids 123, all digits after 1 are in decreasing order.  Combining these properties, $\pi = J_{k} \ominus \left(1 \oplus J_{n-k-1}\right)$ for some $0 \leq k \leq n-1$. 

There are two possible cases for any parking function associated with $\pi=1 \oplus J_{n-1}$:
\begin{enumerate}[(i)]
    \item It has the form $\{1\}|\{n\}|\{n-1\}| \cdots |\{2\}$.
    \item It has the form $\{1,n\}|\{n-1\}| \cdots|\emptyset|\cdots |\{2\}$.
    \end{enumerate}
    There are $n$ parking functions corresponding to $1\oplus J_{n-1}$; one of the first type, and $n-1$ of the second type, depending on the location of the empty block.
    
    Similarly, there are two possible cases for any parking function associated with $\pi=J_{k} \ominus \left(1 \oplus J_{n-k-1}\right)$ with $1 \leq k \leq n-1$:
    \begin{enumerate}[(i)]
    \item It has the form $\{n\}|\cdots |\{n-k+1\}|\{1\}|\{n-k\}| \cdots |\{2\}$.
    \item It has the form $\{n\}|\cdots |\{n-k+1\}|\{1,n-k\}|\{n-k-1\}| \cdots|\emptyset|\cdots |\{2\}$; after choosing what $k$ is, there are $(n-k-1)$ positions for the empty block and this gives a total of $(n-k-1)$ possibilities for each of the $(n-2)$ choices of $k$. 
\end{enumerate}

For parking functions corresponding to $J_{k} \ominus \left(1 \oplus J_{n-k-1}\right)$, there are  $n-1$ parking functions of the first type, depending on the choice of $k$.  Finally, for the second case, given a specific value of $k$, there are $n-k-1$ positions for the empty block.

Hence, the total number of possible parking functions is given by:
$$n+(n-1)+\sum_{k=1}^{n-1} (n-k-1) =  \binom{n+1}{2}.$$

Finally, suppose $\pi \in \mathcal{S}_n(123,231,312)$.  Suppose $\pi_k=1$.  Since $\pi$ avoids 312, all digits before 1 must be smaller than all digits after 1.  Since $\pi$ avoids 231, all digits before 1 must be in decreasing order.  Since $\pi$ avoids 123, all digits after 1 must be in decreasing order.  Combining all these properties, we see that $\pi=J_k \oplus J_{n-k}$ for some $1 \leq k \leq n$.

 To count the number of parking functions associated with $\pi$, we will establish a bijection between $\mathcal{S}_n(123, 231, 312)$ and $\mathcal{S}_n(123, 132, 312)$ that preserves the location of ascents and descents. Suppose $\pi \in \mathcal{S}_n(123, 231, 312)$ with $\pi_k=1$.  Increment all digits before $1$ by $n-k$ and decrement all digits after 1 by $k-1$ to obtain a permutation of the form $J_{k-1} \ominus \left(1 \oplus J_{n-k}\right)$.  In particular the only ascent in either permutation is between the digits $\pi_k=1$ and $\pi_{k+1}$.  Therefore, any parking function associated with one permutation corresponds to a parking function associated with the other permutation by making the appropriate adjustment of all digits other than 1.  Since we have a bijection between parking functions avoiding $\{123, 132, 312\}$ and parking functions avoiding $\{123, 231, 312\}$, they have the same enumeration.
\end{proof}

\begin{theorem}
\label{T:123_213_312}
$$\mathrm{pf}_n(123,213,312)=n(n-1)+1.$$
\end{theorem}

\begin{proof}
Let $\pi \in \mathcal{S}_n(123,213,312)$. 
Since $\pi$ avoids 123 and 213, either $\pi_1=n$ or $\pi_2=n$.

In the first case, since $\pi$ avoids 321, $\pi=J_n$.  In the second case, since $\pi$ avoids 312, all digits after $n$ are in decreasing order, and we only have a choice of the value of $\pi_1$.  Therefore, in this case, $\pi = \left(1 \oplus J_{n-k}\right) \ominus J_{k-1}$, where $1 \leq k \leq n-1$.

There is one parking function corresponding to $J_n$.

On the other hand, given a choice of $1 \leq k \leq n-1$, there are $n$ parking functions corresponding to $\left(1 \oplus J_{n-k}\right) \ominus J_{k-1}$.  One parking function has all blocks of size 1, while the other $n-1$ parking functions begin with the block $\{k,n\}$.  Since there are no other ascents in the underlying permutation, all blocks after the inital block must have size at most 1.  Then there are $n-1$ choices for the location of the empty block.  

These cases give a total of $1+n(n-1)$ pattern-avoiding parking functions.
\end{proof}

\begin{theorem}
\label{T:123_132_213}
$\mathrm{pf}_n(123,132,213)$ is equal to the number of rooted ordered trees with $n+1$ edges such that every vertex is either a leaf or adjacent to a leaf.
\end{theorem}

Theorem \ref{T:123_132_213} is our first result without an explicit formula.  These trees are counted in sequence A143363 of the On-Line Encyclopedia of Integer Sequences \cite{oeis}, and they were first studied by Cheon and Shapiro \cite{CS08}.  The exact enumeration is best given by a functional equation satisfied by the relevant generating function.  For our present context, it is of greater interest to show the parallel structure between these trees and the pattern-avoiding parking functions.  A reader interested in the exact enumeration may consult \cite{CS08} or \cite{oeis}.

\begin{proof}[Proof of Theorem \ref{T:123_132_213}]
Suppose $\pi \in \mathcal{S}_n(123, 132, 213)$.  Since $\pi$ avoids 123 and 213 either $\pi_1=n$ or $\pi_2=n$.  If $\pi_2=n$, then since $\pi$ avoids 132, $\pi_1=n-1$. In either case the rest of the permutation can be filled in recursively.  Therefore, if $\pi \in \mathcal{S}_n(123,132,213)$, then either $\pi=1 \ominus \pi^\prime$ where $\pi^\prime \in \mathcal{S}_{n-1}(123,132,213)$ or $\pi=12 \ominus \pi^\prime$ where $\pi^\prime \in \mathcal{S}_{n-2}(123,132,213)$.  Due to this structure, the number of such permutations is counted by the Fibonacci numbers.  Also, observe that if $\pi \in \mathcal{S}_n(123, 132, 213)$, there are never two consecutive ascents in $\pi$.    This implies that in any parking function avoiding $\{123, 132, 213\}$, every block has size at most 2.

We will give a recursive bijection $f$ between the trees described in the theorem statement and the set of $\{123, 132, 213\}$-avoiding parking functions.  For the trees, we will call an edge containing a leaf an \emph{external edge} and an edge between two internal vertices an \emph{internal edge}.  Given a tree $T$ where every vertex is either a leaf or adjacent to a leaf, let $f(T)$ be the corresponding parking function.  In our bijection, if $T$ has $i$ internal edges then the associated permutation for $f(T)$ has $i$ ascents.

As base cases, we set:

\begin{center}

    $f\left(\raisebox{-0.1in}{\begin{tikzpicture}
\node[circle,draw,minimum size=0.2cm] (1) at (1,1) {};
\node[circle,draw,minimum size=0.2cm] (2) at (1,0) {};
\draw (1)--(2);
\end{tikzpicture}}\right) = \emptyset$

    $f\left(\raisebox{-0.1in}{\begin{tikzpicture}
\node[circle,draw,minimum size=0.2cm] (1) at (1,1) {};
\node[circle,draw,minimum size=0.2cm] (2) at (0.5,0) {};
\node[circle,draw,minimum size=0.2cm] (3) at (1.5,0) {};
\draw (1)--(2);
\draw (1)--(3);
\end{tikzpicture}}\right) = \{1\}$

\end{center}

For the general case, we consider two cases: either the tree has no internal edges or it has at least one internal edge.

If the tree has no internal edges, then every edge has the root as one of its endpoints.  This tree with $n+1$ edges corresponds to the unique parking function associated with the decreasing permutation of length $n$.  (Notice that the $n=0$ and $n=1$ examples of this case are exactly the two base cases shown above.)

Now, consider a tree with at least one internal edge.   Suppose the root is contained in $m$ edges and $\ell$ of them are internal, called $e_1, e_2, \dots, e_\ell$.  There are $\binom{m}{\ell}$ ways these $\ell$ edges can be distributed in order among the $m$ edges containing the root, so let their respective positions be $p_1,\dots, p_\ell$ where $\{p_1,\dots,p_\ell\} \subseteq [m]$. Since $e_1, e_2, \dots, e_\ell$ are internal edges, they each have a rooted ordered tree below them.  Call these trees $t_1,\dots, t_\ell$.

Since we are forming bijection $f$ recursively, each of $t_1,\dots, t_\ell$ corresponds to a pattern-avoiding parking function.  If $t_i$ has $j+1$ edges, then $f(t_i)$ corresponds to a parking function of length $j$.  Accordingly, let $j=\left|t_i\right|$.

To complete the bijection, we consider (a) a parking function $\hat{p}$ corresponding to the set of edges containing the root, and (b) an algorithm to incorporate $\hat{p}$ and the parking functions corresponding to the $f(t_i)$'s into one single parking function $f(T)$. 

For $\hat{p}$, create a parking function with underlying permutation $$\underbrace{12 \ominus \cdots \ominus 12}_{\ell \text{ times}} \ominus \underbrace{1\ominus \cdots \ominus 1}_{m-\ell-1 \text{ times}},$$ where $\ell$ corresponds to the number of internal edges out of the $m$ total edges containing the root of $T$. The parking function $\hat{p}$ will only contain empty blocks after the last ascent bottom.

If $p_1=1$, the first $\ell-1$ ascents appear in blocks of size 2 while the last ascent appears in two blocks of size 1, leaving $m-\ell$ blocks of size 1 after the last ascent bottom.  There are $\binom{m-1}{\ell-1}$ ways to distribute $\ell-1$ empty blocks after the last ascent bottom.

If $p_1>1$, then all ascents appear in blocks of size 2, and there are $m-\ell-1$ blocks of size 1.  There are $\binom{m-1}{\ell}$ ways to distribute $\ell$ empty blocks after the last ascent bottom.

In both cases, the locations of the internal edges among the final $m-1$ edges correspond to the locations of the empty blocks at the end of the parking function.  

As some small examples, here are the simplest trees with one internal edge.  Both of these have one internal edge and one external edge descending from the root so $m=2$ and $\ell=1$.  The first tree has $p_1=1$ while the second has $p_1=2$.  Both correspond to the permutation 12, but the location of the internal edge determines whether there is an empty block and where it is.

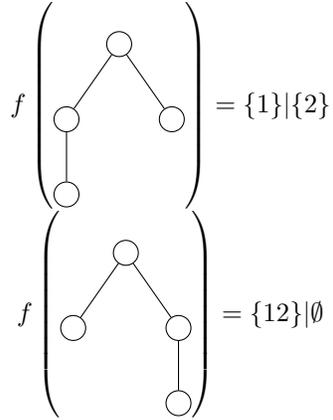
\begin{center}
        $f\left(\raisebox{-0.5in}{\begin{tikzpicture}
\node[circle,draw,minimum size=0.2cm] (1) at (1,1) {};
\node[circle,draw,minimum size=0.2cm] (2) at (0.3,0) {};
\node[circle,draw,minimum size=0.2cm] (3) at (0.3,-1) {};
\node[circle,draw,minimum size=0.2cm] (4) at (1.7,0) {};
\draw (1)--(2);
\draw (2)--(3);
\draw (1)--(4);
\end{tikzpicture}}\right) = \{1\} \vert\{2\}$
    
    $f\left(\raisebox{-0.5in}{\begin{tikzpicture}
\node[circle,draw,minimum size=0.2cm] (1) at (1,1) {};
\node[circle,draw,minimum size=0.2cm] (2) at (0.3,0) {};
\node[circle,draw,minimum size=0.2cm] (3) at (1.7,-1) {};
\node[circle,draw,minimum size=0.2cm] (4) at (1.7,0) {};
\draw (1)--(2);
\draw (4)--(3);
\draw (1)--(4);
\end{tikzpicture}}\right) = \{12\}\vert \emptyset$

\end{center}

Finally, we are ready to incorporate trees $t_1,\dots, t_\ell$.  If $T$ has internal edges, we first form a parking function $\hat{p}$ with $\ell$ ascents followed by $m-\ell-1$ entries in decreasing order, with empty blocks as prescribed above.  We also recursively determine parking functions $f(t_1),f(t_2),\dots, f(t_\ell)$ of lengths $\left|t_1\right|,\left|t_2\right|,\dots, \left|t_\ell\right|$ respectively.

Now, increment the digits in the $k$th ascent of $P$ by $\sum_{i=k+1}^{\ell} \left|t_i\right|$.  Also, increment the digits in $f(t_k)$ by $2(\ell-k+1) +(m-\ell-1)+\sum_{i=k+1}^{\ell} \left|t_i\right|$.  Insert this incremented version of $f(t_i)$ immediately before the $k$th ascent of $P$.

\end{proof}

\begin{figure}[hbt]
\begin{center}\begin{tikzpicture}
\node[circle,draw,minimum size=0.2cm] (1) at (1,1) {};
\node[circle,draw,minimum size=0.2cm] (2) at (0.3,0) {};
\node[circle,draw,minimum size=0.2cm] (3) at (1.7,0) {};
\node[circle,draw,minimum size=0.2cm] (4) at (1.7,-1) {};
\node[circle,draw,minimum size=0.2cm] (5) at (1.2,-1) {};
\node[circle,draw,minimum size=0.2cm] (6) at (2.2,-1) {};
\node[circle,draw,minimum size=0.2cm] (7) at (-0.4,0) {};
\node[circle,draw,minimum size=0.2cm] (8) at (2.4,0) {};
\node[circle,draw,minimum size=0.2cm] (9) at (0.1,-1) {};
\node[circle,draw,minimum size=0.2cm] (10) at (0.5,-1) {};
\node[circle,draw,minimum size=0.2cm] (11) at (1.2,-2) {};
\draw (1)--(2);
\draw (1)--(7);
\draw (1)--(8);
\draw (3)--(4);
\draw (3)--(5);
\draw (3)--(6);
\draw (1)--(3);
\draw (2)--(9);
\draw (5)--(11);
\draw (2)--(10);
\end{tikzpicture}
\end{center}
\caption{A rooted ordered tree where every vertex is a leaf or adjacent to a leaf}
\label{F:thm11eg}
\end{figure}

Before moving to a new result, we consider a concrete example of the bijection in the proof of Theorem \ref{T:123_132_213}. Consider the tree shown in Figure \ref{F:thm11eg}.  The root has 4 children so $m=4$ and $\ell=2$ with $p_1=2$ and $p_2=3$.  Therefore, parking function $\hat{p}$, corresponding to the set of edges containing the root has underlying permutation $12\ominus 12 \ominus 1 = 45231$.  Since $p_1>1$, we see that 45 will appear in the same block and 23 will appear in the same block.  There are two empty blocks that must appear after the 23, and since the internal edges are the first two among the last $m-1$ edges descending from the root, so parking function $\hat{p}$ has empty blocks as the first two blocks after $\{2,3\}$, i.e.,  $\hat{p}=\{4,5\}\vert \{2,3\}\vert \emptyset \vert \emptyset \vert \{1\}$.

We know that $f(t_1)=\{1\}$ as a base case above.

We also need to determine $f(t_2) = f\left(\raisebox{-0.5in}{\begin{tikzpicture}
\node[circle,draw,minimum size=0.2cm] (3) at (1.7,0) {};
\node[circle,draw,minimum size=0.2cm] (4) at (1.7,-1) {};
\node[circle,draw,minimum size=0.2cm] (5) at (1.2,-1) {};
\node[circle,draw,minimum size=0.2cm] (6) at (2.2,-1) {};
\node[circle,draw,minimum size=0.2cm] (11) at (1.2,-2) {};
\draw (3)--(4);
\draw (3)--(5);
\draw (3)--(6);
\draw (5)--(11);
\end{tikzpicture}}\right)$.

By a similar process, we see the set of edges containing the root of $t_2$ has $m=3$ and $\ell=1$, and therefore has underlying permutation 231.  Since the internal edge is first, we have 2 and 3 in separate blocks, so $f(t_2) = \{2\}\vert\{3\}\vert \{1\}$.

Now, we increment.  $\left|t_1\right|=1$ and $\left|t_2\right|=3$, so the digits in $f(t_1)$ get incremented by $\left|t_2\right|+2(\ell-1+1) +(m-\ell-1)= 3+2(2)+1$ to become $\{9\}$.  The digits in $f(t_2)$ get incremented by $2(\ell-2+1) +(m-\ell-1)= 2(1)+1$ to become $\{5\}\vert\{6\}\vert \{4\}$.  Similarly, digits 45 in the original generation get incremented by 3 to become 78.

This example tree corresponds to the parking function 
$$\textcolor{blue}{\{9\}}\vert \{7,8\} \vert\textcolor{red}{\{5\}\vert\{6\}\vert \{4\}} \vert\{2,3\}\vert \emptyset \vert \emptyset \vert\{1\}.$$

The blue font highlights $f(t_1)$, the red font highlights $f(t_2)$, and the black font highlights the components of the parking function that come from the set of edges containing the root of the original tree.

\begin{theorem}
\label{T:132_231_312}
$$\mathrm{pf}_n(132,213,231)=\mathrm{pf}_n(132,231,312)=\sum_{i=1}^n C_i.$$
\end{theorem}

\begin{proof}
Suppose $\pi \in \mathcal{S}_n(132, 213, 231)$.  Since $\pi$ avoids 132 and 231, either $\pi_1=n$ or $\pi_n=n$.  If $\pi_n=n$, then since $\pi$ avoids 213, $\pi=I_n$.  On the other hand, if $\pi_1=n$, we may recursively fill in $\pi_2\cdots \pi_n$ with any member of $\mathcal{S}_{n-1}(132, 213, 231)$, so $\pi=J_k\ominus I_{n-k}$ for some $0 \leq k \leq n$.

Suppose $\pi \in \mathcal{S}_n(132, 231, 312)$.  Since $\pi$ avoids 132 and 231, either $\pi_1=n$ or $\pi_n=n$.  If $\pi_1=n$, then since $\pi$ avoids 312, $\pi = J_n$.  On the other hand, if $\pi_n=n$, we may recursively fill in $\pi_1 \cdots \pi_{n-1}$ with any member of $\mathcal{S}_{n-1}(132, 231, 312)$, so $\pi = J_k \oplus I_{n-k}$ for some $0 \leq k \leq n$.

For each of patterns, the pattern-avoiding permutations we are concerned with consist of $d$ descents followed by $n-1-d$ ascents, and there is one pattern-avoiding permutation for each choice of $0 \leq d \leq n-1$.  The initial $d+1$ digits appear in decreasing order in $\pi$ and thus they must appear in blocks of size 1 at the beginning of any corresponding parking function.  However, the remaining $n-(d-1)$ digits appear in increasing order in $\pi$ and correspond to $C_{n-d+1}$ ways to complete the final $n-d+1$ blocks of the parking function with increasing labels.

Summing over all possible values of $d$, gives $\sum_{d=1}^n C_{n-d+1}$ pattern-avoiding parking functions.
\end{proof}

\begin{theorem}
\label{T:132_213_312}
$$\mathrm{pf}_n(132,213,312)=\mathrm{pf}_n(213,231,312)=C_{n+1}-C_n.$$
\end{theorem}

\begin{proof}
Suppose $\pi \in \mathcal{S}_n(132, 213, 312)$.  Since $\pi$ avoids 132, all digits before $n$ are larger than all digits after $n$.  Since $\pi$ avoids 213 and 312, all digits before $n$ are in increasing order and all digits after $n$ are in decreasing order.  Combining these properties, $\pi = I_k \ominus J_{n-k}$ for some $0 \leq k \leq n$.  Similarly, any permutation avoiding 213, 231, and 312 has the form $I_{k-1} \oplus J_{n-k+1}$ and thus also begins with an increasing sequence of length $k$. 
In both cases, the pattern-pattern avoiding permutation in question begin with a sequence of $k-1$ ascents followed by $n-k$ descents which implies $\mathrm{pf}_n(132,213,312)=\mathrm{pf}_n(213,231,312)$.

Let $a(n,k)$ be the number of $\{213,231,312\}$-avoiding parking functions of size $n$ with exactly $k-1$ ascents.  This is a refinement of the enumeration we are interested in, and we will prove that $a(n,k)$ is given by triangle A030237 in the On-Line Encyclopedia of Integer Sequences.  In particular, we will show that $a(n,1)=1$ and $a(n,n)=C_n$ for all $n \geq 0$ and for $1<k<n$, $a(n,k)=a(n,k-1)+a(n-1,k)$.

When $k=1$, the unique permutation with no ascents is $J_n$, and there is a unique way to write the decreasing permutation as a parking function, so, indeed $a(n,1)=1$.

Similarly, when $k=n$, the unique permutation with $n-1$ ascents is $I_n$, and there are $C_n$ ways to write the increasing permutation as as a parking function, so, $a(n,n)=C_n$.

Now, consider the situation where $1<k<n$.  $a(n,k)$ counts parking functions corresponding to permutations of the form $I_{k-1} \oplus J_{n-k+1}$, so $a(n,k)$ counts parking functions corresponding to permutations ending in $k$.  We consider two cases: either the last block of a parking function counted by $a(n,k)$ is $\{k\}$ or it is empty.

In the first case, delete the last block.  Then subtract 1 from each digit larger than $k$.  This produces a pattern-avoiding parking function of size $n-1$ ending in $k$, and there are $a(n-1,k)$ such parking functions.

In the second case, consider the Dyck path representation of the parking function and find the north-step labeled $k-1$.  There must be $j$ east-steps between this north-step and the previous north-step where $j \geq 0$.  Remove these east-steps and the north-step labeled $k-1$ (deleting a subpath of the form $E^jN$ and insert them immediately before the last east-step as a $NE^j$ subpath to get a parking function counted by $a(n,k-1)$.  The final empty set block guarantees that this operation still produces a path that stays above $y=x$.

Together these cases show that $a(n,k)=a(n-1,k)+a(n,k-1)$, which is exactly the recurrence that characterizes Catalan's triangle.  Since we partitioned our parking functions according to the longest increasing prefix, we are interested in the row sums of this triangle.  However, we only have $1 \leq k \leq n$, so our rows have $n$ entries and are missing the final entry of $C_n$ in each row of A030237.  Since it is known that the row sums of Catalan's triangle are given by $C_{n+1}$, the missing diagonal guarantees that our row sums are given by $C_{n+1}-C_n$.
\end{proof}

We present a few examples to illustrate the recurrence in the proof of Theorem \ref{T:132_213_312}.

Consider $\{1,2\}\vert\{3\}\vert\emptyset \vert\{5\}\vert\{4\}$, which is counted by $a(5,4)$.  If we delete the final block and then subtract 1 from each entry larger than 4, we obtain $\{1,2\}\vert\{3\}\vert\emptyset \vert\{4\}$ which is counted by $a(4,4)$, as desired.

On the other hand, consider $\{1,2\}\vert\{3\}\vert\{5\}\vert\{4\}\vert\emptyset$, which is also counted by $a(5,4)$.  Since the final block is empty, we consider the Dyck path corresponding to this parking function, which has the form NN\textbf{EN}ENENEE.  The third north-step is the one labeled with $k-1=3$, and it is preceded by one east-step.  These steps are highlighted with bold font above.  We move them prior to the final east-step to get NNENENE\textbf{NE}E.  Since the label 3 remains on the bolded north-step, this corresponds to the parking function $\{1,2\}\vert\{5\}\vert\{4\}\vert\{3\}\vert\emptyset$, which is counted by $a(5,3)$

Similarly, consider $\{1,2\}\vert\emptyset\vert\{3,5\}\vert\{4\}\vert\emptyset$, which is also counted by $a(5,4)$.  Since the final block is empty, we again consider the Dyck path corresponding to this parking function, which has the form NN\textbf{EEN}NENEE.  The third north step is the one labeled with $k-1=3$, and it is preceded by two east-steps.  Deleting these three steps and inserting them in reverse order before the final E gives NNNENE\textbf{NEE}E, which corresponds to the parking function $\{1,2,5\}\vert\{4\}\vert\{3\}\vert\emptyset\vert\emptyset$, counted by $a(5,3)$.

\begin{theorem}
\label{T:132_231_321}
$$\mathrm{pf}_n(132,231,321)=C_n+(n-1)C_{n-1}.$$
\end{theorem}

\begin{proof}
Suppose $\pi \in \mathcal{S}_n(132, 231, 321)$.  Since $\pi$ avoids 132 and 231, either $\pi_1=n$ or $\pi_n=n$.  In the first case, since $\pi$ avoids 321, $\pi=1 \ominus I_{n-1}$.  On the other hand if $\pi_n=n$, we may recursively fill in the digits $\pi_1\cdots \pi_{n-1}$ with any member of $\mathcal{S}_{n-1}(132, 231, 321)$.  Together these results tell us that $\pi=\left(1 \ominus I_{k-1}\right)\oplus I_{n-k}$, where $1 \leq k \leq n$.

In the situation where $k=1$ (i.e. where $\pi=I_n$), there are $C_n$ possible parking functions.

On the other hand, if $k>1$, then $\pi$ begins with a descent, so the first digit must appear in its own block.  However, the remaining $n-1$ digits of $\pi$ appear in increasing order and must be distributed among $n-1$ blocks in one of $C_{n-1}$ possible ways.  Since there are $(n-1)$ possible values for $k$ when $k>1$, we get $(n-1)C_{n-1}$ parking functions in this case.
\end{proof}

\begin{theorem}
\label{T:132_213_321}
$$\mathrm{pf}_n(132,213,321)=\mathrm{pf}_n(132,312,321)=\mathrm{pf}_n(213, 231,321)=\binom{2n-1}{n}.$$
\end{theorem}

\begin{proof}
If $\pi \in \mathcal{S}_n(132, 213, 321)$ then $\pi=I_{k} \ominus I_{n-k}$.  If $\pi \in \mathcal{S}_n(132,312,321)$ then $\pi=\left(I_{k}\ominus 1\right) \oplus I_{n-k-1}$.  If $\pi \in \mathcal{S}_n(213, 231,321)$ then $\pi=I_{k-1} \oplus \left(1 \ominus I_{n-k}\right)$.  In all three situations, we consider permutations with at most 1 descent.  Further, the location of the descent uniquely describes the permutation.  For this reason, we know
$$\mathrm{pf}_n(132,213,321)=\mathrm{pf}_n(132,312,321)=\mathrm{pf}_n(213, 231,321),$$ and it remains to enumerate one of these directly.

Let $\pi \in \mathcal{S}_n(132, 213, 321)$. We know that $\pi = I_{k} \ominus I_{n-k}$ for some $k$, where $1 \leq k \leq n$.  Now, we consider the Dyck path representation of a parking function that may be associated with $\pi$.  Since there is a descent immediately after the digit $n$ in $\pi$, we know the north step labeled with $n$ is immediately followed by an east step. Call this north step $N^*$ and the subsequent east step $E^*$.   Suppose that there are $i$ east steps before $N^*$ where $0 \leq i \leq k-1$.  There are $C_i$ ways to arrange these east steps with their corresponding parentheses-matching north steps.  Similarly, there are $n-1-i$ east steps after $E^*$ and there are $C_{n-i-1}$ ways to arrange these east steps with their corresponding parentheses-matching north steps (some of which might appear prior to $N^*$).

Summing over all possible values of $i$ and $k$ and simplifying with a CAS gives

$$\sum_{k=1}^n \sum_{i=0}^{k-1} C_i C_{n-i-1} = \binom{2n-1}{n}.$$

\end{proof}

\begin{theorem}
\label{T:213_312_321}
$$\mathrm{pf}_n(213,312,321)= C_n+(n-1)(C_n-C_{n-1})=nC_n-(n-1)C_{n-1}.$$
\end{theorem}

\begin{proof}
Suppose $\pi \in S_n(213, 312, 321)$. Since $\pi$ avoids 312 and 321, $\pi_1 =1$ or $\pi_1=2$; otherwise $\pi_1$, 1, and 2 would form either a 312 or a 321 pattern.  If $\pi_1=2$, since $\pi$ avoids 213, we know $\pi_n=1$.  Now, since $\pi_n=1$ and $\pi$ avoids 321, it must be the case that $\pi=I_{n-1} \ominus 1$.  On the other hand if $\pi_1=1$, we may recursively fill in $\pi_2\cdots \pi_n$ with any member of $S_{n-1}(213, 312, 321)$ (with each of its digits incremented by 1.  Therefore, if $\pi \in S_n(213, 312, 321)$, then $\pi = I_{k-1} \oplus \left(I_{n-k} \ominus 1\right)$, where $1 \leq k \leq n$.

When $k=n$, $\pi = I_n$ and there are $C_n$ parking functions associated with $\pi$.  When $k<n$, we count the permutations associated with $I_n$ where $n-1$ and $n$ appear in separate blocks.  We can transform any such parking function into a parking function associated to $I_{k-1} \oplus \left(I_{n-k} \ominus 1\right)$ by relabeling $n$ (in its own block) as $k$, and by relabeling every digit $i$ with $k \leq i \leq n-1$ as $i+1$. This transformation can be easily reversed.

Thus, to count the parking functions associated with $I_n$ where $n-1$ and $n$ appear in separate blocks, we need only subtract the parking functions where $n-1$ and $n$ appear in the same block, which is given by $C_{n-1}$, to get a total of $C_n-C_{n-1}$ such parking functions. Summing over all possible values for $k<n$ gives $(n-1)(C_n-C_{n-1})$ parking functions.  Combining this with the $k=n$ case proves the theorem.
\end{proof}

\begin{theorem}
\label{T:231_312_321}
$$\mathrm{pf}_n(231,312,321)=\sum_{k=0}^{\lfloor\frac{n}{2}\rfloor} \frac{1}{n+1}\binom{2n-k}{n+k}\binom{n+k}{k}.$$
\end{theorem}

\begin{proof}
The expression given above is known to be the number of dissections of a convex (n+2)-gon into triangles and quadrilaterals by nonintersecting diagonals (see OEIS A001002 and \cite{M48}).  We will prove this theorem by giving a bijection between $\{231,312,321\}$-avoiding parking functions and these disections.

First, notice that the number of parking functions corresponding to the increasing permutation is given by the $n$th Catalan number, as is the number of dissections of an $(n+2)$-gon into triangles.  We first give a bijection between these two special cases of the sets in question.

The orientation of the polygon partitions matters; in other words, each polygon can be considered to have a specific distinguished edge along its perimeter.  Call the triangle that uses this edge the distinguished triangle.  When $n=1$, the polygon is a triangle.  This corresponds to the one parking function of size 1.

When $n>1$, we consider 3 cases: 
\begin{enumerate}[(i)]
\item The distinguished edge is the first of two consecutive distinguished triangle edges reading clockwise around the perimeter of the polygon.
\item The distinguished edge is the second of two consecutive distinguished triangle edges reading clockwise around the perimeter of the polygon.
\item the distinguished edge is the only edge of the distinguished triangle on the perimeter of the polygon.
\end{enumerate}
This orientation of the distinguished triangle tells us about the relationship between $n$ and $n-1$ in the corresponding parking function.  

In case (i), $n-1$ and $n$ appear in different but consecutive blocks.  We may remove the distinguished triangle to obtain a triangulation of an $(n+1)$-gon.  The edge that was formerly part of the distinguished triangle is now the distinguished edge and we proceed recursively.

In case (ii), $n-1$ and $n$ appear in the same block, and an empty block appears at the end of the parking function.  We may remove the distinguished triangle to obtain a triangulation of an $(n+1)$-gon.  The edge that was formerly part of the distinguished triangle is now the distinguished edge and we proceed recursively.

Notice that in a triangulation with only case (i) and case (ii) triangles, the parking function has all of the digits $1,\dots,n$ in a sequence of consecutive blocks, so all empty blocks come after all non-empty blocks.

In case (iii), $n-1$ and $n$ appear in different blocks with at least one empty block between them.  We may remove the distinguished triangle to obtain a two smaller triangulated polygons -- one to the left of the distinguished triangle and one to the right.  For each of these smaller polygons, the edge that was formerly part of the distinguished triangle is now the distinguished edge and we proceed recursively. Call the parking function corresponding to the right polygon $P_1$ and the parking function corresponding to the left polygon $P_2$.  If there are $k$ digits in $P_1$, add $k$ to every digit in $P_2$ so that it uses the digits $k+1, \dots, n-1$.  $P_1$ can be partitioned into $P_1^{\prime}$ and $P_1^{\prime\prime}$ where $P_1^{\prime}$ is the initial blocks of $P_1$, ending with the last non-empty block, and $P_1^{\prime \prime}$ consists of any empty blocks at the end of $P_1$.  We construct the larger parking function as follows: Take all blocks of $P_1^{\prime}$ and merge the last block of $P_1^{\prime}$ with the first block of $P_2$, follow with all of $P_2$, an empty block, a block containing $n$, and then $P_1^{\prime \prime}$.  In other words the empty block between $n-1$ and $n$ arises because $k$ and $k+1$ appear in the same block earlier in the parking function.

Now that we have a bijection between triangulations of the $(n+2)$-gon and parking functions corresponding to the increasing permutation, we address the more general situation of the theorem.

First observe that permutations that avoid 231, 312, and 321 are by definition direct sums of 1 and 21 permutations.  We claim that the number of quadrilaterals in a partition of an $(n+2)$-gon corresponds to the number of consecutive 21 patterns in the underlying permutation of the corresponding parking function.  In particular, consider a partition of an $(n+2)$-gon into triangles and quadrilaterals.  There is a unique way to partition each quadrilateral into two triangles so that we always see case (i) and case (iii) triangles arise from the quadrilateral.  In particular, when we arrive at the quadrilateral in our recursive reading, draw the diagonal that uses the left endpoint of the distinguished edge on the perimeter of the polygon.  Now, find the parking function that corresponds to the triangulation.  By construction, the two digits used in any quadrilateral appear in separate blocks, and so we may transpose them from a 12 pattern into a 21 pattern while still having a legal parking function.

\end{proof}

\section{Avoiding two patterns}

There are $\binom{6}{2}=15$ distinct pairs of patterns of length 2. The pair $\{123,321\}$ is already addressed by the Erd\H{o}s-Szekeres Theorem.  From the remaining 14 pairs of patterns, we get 11 different enumeration sequences.  A summary of results is shown in Table \ref{tab:two_patterns}.

\begin{table}[hbt]
    \centering
    \begin{tabular}{|c|c|c|c|}
    \hline
    Patterns $P$&$\mathrm{pf}_n(P)$, $1 \leq n \leq 6$&OEIS&Result\\
    \hline
123, 231&1, 3, 8, 17, 31, 51&A105163&Theorem \ref{T:123_231}\\
\hline
123, 312&1, 3, 9, 21, 41, 71&A064999&Theorem \ref{T:123_312}\\
\hline
123, 132&1, 3, 8, 24, 75, 243&A000958 (conjectured)&Theorem \ref{T:123_132}\\
\hline
123, 213&1, 3, 9, 28, 90, 297&A000245&Theorem \ref{T:123_213}\\
\hline
132, 231&1, 3, 10, 36, 137, 543&A002212&Theorem \ref{T:132_231}\\
\hline
132, 213&\multirow{4}{*}{1, 3, 11, 45, 197, 903}&\multirow{4}{*}{A001003}&\multirow{4}{*}{Theorem \ref{T:132_213}}\\
132, 312&&&\\
213, 231&&&\\
231, 312&&&\\
\hline
132, 321&1, 3, 12, 52, 229, 1006&new&Theorem \ref{T:132_321}\\
\hline
213, 321&1, 3, 13, 60, 275, 1238&new&Theorem \ref{T:213_321}\\
\hline
213, 312&1, 3, 12, 54, 259, 1293&new&Theorem \ref{T:213_312}\\
\hline
231, 321&1, 3, 12, 55, 273, 1428&A001764&Theorem \ref{T:231_321}\\
\hline
312, 321&1, 3, 13, 63, 324, 1736&new&Theorem \ref{T:312_321}\\
\hline
    \end{tabular}
    \caption{\textit{Enumeration data for parking functions avoiding a pair of patterns of length 3}}
    \label{tab:two_patterns}
\end{table}

\begin{theorem}
\label{T:123_231}
$$\mathrm{pf}_n(123,231)=\binom{n+1}{3}+\binom{n}{2}+1.$$
\end{theorem}

\begin{proof}
Suppose $\pi \in \mathcal{S}_n(123,231)$.  Either $\pi_1=n$ or $\pi_i=n$ whre $i \geq 2$.  In the first case $\pi=1 \ominus \pi^{\prime}$ where $\pi^{\prime} \in \mathcal{S}_{n-1}(123,231)$.  In the second case, since $\pi$ avoids 231, all digits before $n$ must be smaller than all digits after $n$, and since $\pi$ avoids 123, all digits before $n$ must appear in decreasing order.  However, since $\pi$ avoids 123, this also means that all digits after $\pi_{i-1}=1$ must also appear in decreasing order, and thus $\pi = J_{i-1} \oplus J_{n-i+1}$.  Putting these two cases together, we see that any permutation that avoids both 123 and 231 has the form $J_a \ominus (J_b \oplus J_{n-a-b})$ where $a,b \geq 0$.  If $b=0$, there is 1 possible permutation, i.e. $J_n$.  If $b>0$, then there are $\binom{n}{2}$ ways to choose the positions that begin and end the $J_b$ subpermutation, for a total of $\binom{n}{2}+1$ such permutations.

Now we consider how many parking functions correspond to these permutations.  When $b=0$, we have the decreasing permutation which must be written on the Dyck path $(NE)^n$, so there is only one possible parking function in this case.

If $b>0$ only ascent in the permutation is between the last digit of $J_b$ and the first digit of $J_{n-a-b}$.  Either all digits are written in blocks of size 1, or the digits of this ascent appear in one block of size 2, and an empty block appears later.  If all blocks are of size 1, we need only know $a$ and $b$ to uniquely determine the corresponding parking function.  There are $\binom{n}{2}$ choices for $a$ and $b$ in this situation, since choosing the position where the $J_a$ subsequence ends and then choosing the position where the $J_b$ subsequence ends uniquely determines their lengths.

In the situation where there is a block of size 2, we must still determine the location for the end of the $J_a$ subsequence and the position for the end of the $J_b$ subsequence, but now we must also determine the position of the empty block.  There are $\binom{n+1}{3}$ ways to choose such a combination.

\end{proof}

\begin{theorem}
\label{T:123_312}
$$\mathrm{pf}_n(123,312)=2\binom{n+1}{3}+1.$$
\end{theorem}

\begin{proof}
Suppose $\pi \in \mathcal{S}_n(123,312)$.  Since $\pi$ avoids 123, all digits before $n$ appear in decreasing order, and since $\pi$ avoids 312, all digits after $n$ appear in decreasing order.  Further, all digits before $n$ must be consective, since any instance of $\pi_{i+1}<\pi_{i}-1$ would imply that $\pi_i$ and $\pi_{i+1}$ are the first two digits in a 312 pattern.  In summary, $\pi = (J_a \oplus J_b) \ominus J_{n-a-b}$ with $a,b \geq 0$.

If $a=0$, then we have the decreasing permutation which can only be written in one way as a parking function, written on the Dyck path $(NE)^n$.

If $a\neq 0$, the only ascent begins with the last digit of the $J_a$ subsequence.  This means that any corresponding parking function either has all blocks of size 1, or it has one block of size 2 and an empty block later.  Both situations can be encoded by a permutation of the digits $1,\dots ,n$ and a null character such that the digits $1, \dots, n$ avoid 123 and 312.  If the null character appears between the digits of the sole ascent, then we write a parking function with only blocks of size 1.  If the null character appears later, then the digits of the sole ascent appear in a block of size 2, and the null character represents the location of the empty block

Consider the set $\{1,2,\dots , n+1\}$.  For any choice of 3 digits from this set with $i<j<k$, we can make two parking functions.  In both cases, $i$ is the location of the last digit of the $J_a$ subsequence.  $j$ and $k$ are the locations of the last digit of the $J_b$ subsequence and the null character, which may appear in either order.
\end{proof}

\begin{theorem}
\label{T:123_132}
$$\mathrm{pf}_n(123,132)=\sum_{k=2}^{n+1}b(n,k)$$
where
$$b(n,k)=\begin{cases}
0&k<2,\\
1&k=2 \text{ and } n\in \{1,2\},\\
2&n=3 \text{ and } k=2,\\
2b(n-1,k-1)+\sum_{j=k-1}^{n-1}b(n-2,j)&\text{otherwise}.
\end{cases}$$
\end{theorem}

\begin{proof}
Suppose $\pi\in \mathcal{S}_n(123, 132)$.  Since $\pi$ avoids 132, every digit before $n$ is larger than every digit after $n$.  Further, since $\pi$ avoids 123, any digits before $n$ appear in decreasing order.  Thus, either $\pi_1=n$ or $\pi_1=n-1$.  As a result, any parking function avoiding 123 and 132 has as its first block $\{n\}$, $\{n-1\}$ or $\{n-1,n\}$.

In the first case, removing $\{n\}$ produces a pattern-avoiding parking function of size $n-1$.

In the second case, removing $\{n-1\}$ and replacing $n$ with $n-1$ produces a parking function of size $n-1$.

In the final case, the initial block of size 2 implies that there must be an empty block later in the parking function.  In fact, we may consider the up-step in the corresponding Dyck path representation that has label $n-1$ and find its non-crossing matched east step to locate the empty block.  Removing the initial block of size 2 and the corresponding empty block produces a pattern-avoiding parking function of size $n-2$.  To reverse the process, we only need to know the parking function of size $n-2$ and the location of the empty block.  To this end, given a pattern-avoiding parking function, consider the action of adding a new block of size 2 to the beginning.  We call a location where an empty block could be legally placed in this larger parking function an \emph{active site}.  In a parking function of size $n$ with all blocks of size 1, there are $n+1$ active sites -- the beginning, the end, and between any pair of blocks.  However, if there are already some blocks of size 2 and 0, each block of size 2 has a corresponding block of size 0, and they pair as parentheses, with a block of size 2 corresponding to an open parenthesis and a block of size 0 corresponding to a closed parenthesis.  In this situation, positions inside a matched pair of blocks are not active sites.  Every parking function, thus, has at a minimum 2 active sites (at the beginning and the end) and a maximum of $n+1$ active sites (in the case where all blocks are of size 1).

Let $b(n,k)$ be the number of parking functions of size $n$ with exactly $k$ active sites.  Then, $\mathrm{pf}_n(123,132)=\sum_{k=2}^{n+1}b(n,k)$, as desired.

The one parking function of size 1 has 2 active sites. The parking function $\{1,2\}\vert\emptyset$ has 2 active sites, while the other parking functions of size 2 have 3 active sites, which addresses our base cases.

More generally, we consider the effect of a new first block on active sites.  Any parking function beginning with $\{n\}$ or with $\{n-1\}$ has one more active site than the size $n-1$ parking function obtained by removing the first block.  The behavior of a parking function beginning with $\{n-1,n\}$ is more complex.  Given a parking function of size $n-2$ with $k$ active sites, we can obtain $k$ different parking functions by putting $\{n-1,n\}$ on the front and $\emptyset$ in an active site.  When the $\emptyset$ goes in the first available active site, there are $k+1$ active sites in the new larger parking function; when it goes in the second site, there are $k$ active sites, and so on, until when $\emptyset$ goes in the last possible location, there are 2 active sites in the new larger parking function.

Altogether, this analysis tells us that for $n \geq 3$, $$b(n,k)=2b(n-1,k-1)+\sum_{j=k-1}^{n-1}b(n-2,j).$$
The first term covers the two cases where the parking function begins with a block of size 1.  The latter sum runs over all scenarios for beginning with $\{n-1,n\}$ and inserting $\emptyset$ into an active site in a parking function of size $n-2$.
\end{proof}

Although the recurrence in Theorem \ref{T:123_132} is more complicated than many of the other results we have presented, it has nice structure.  Initial values in the triangle $b(n,k)$ are given in Table \ref{tab:123_132}.  Although this triangle appears to be new to the literature, some interesting patterns appear.  For example, $b(n,n+1)=2^{n-1}$.  We can see this readily from the recurrence, which simplifies to
$$b(n,n+1)=2b(n-1,n)+\sum_{j=n}^{n-1}b(n-2,j)=2b(n-1,n)$$ in this case.  This accounts for the parking functions where every digit is in its own block, and therefore matches the total number of $\{123,132\}-$avoiding permutations of size $n$.  However, we are interested in the row sums.  It turns out that $b(n+2,2) = 2b(n+1,1)+\sum_{j=1}^{n+1}b(n,j) = \sum_{j=1}^{n+1}b(n,j)$, and so $b(n+2,2)=\mathrm{pf}_n(123,132)$.  Interestingly, these values appear to match the values of OEIS sequence A000958 which counts ordered rooted trees with n edges having root of odd degree.  It remains open to find a proof showing these sequences agree for arbitrarily large $n$.

\begin{table}[hbt]
    \centering
    \begin{tabular}{|c||c|c|c|c|c|c|c|c|c|c|}
    \hline
    $n\backslash k$&2&3&4&5&6&7&8&9&10&11\\
    \hline
    1&1&&&&&&&&&\\
    \hline
    2&1&2&&&&&&&&\\
    \hline
    3&1&3&4&&&&&&&\\
    \hline
    4&3&5&8&8&&&&&&\\
    \hline
    5&8&14&17&20&16&&&&&\\
    \hline
    6&24&40&49&50&48&32&&&&\\
    \hline
    7&75&123&147&151&136&112&64&&&\\
    \hline
    8&243&393&465&473&432&352&256&128&&\\
    \hline
    9&808&1294&1519&1540&1409&1176&880&576&256&\\
    \hline
    10&2742&4358&5087&5144&4721&3986&3088&2144&1280& 512\\
    \hline
    \end{tabular}
    \caption{Values of $b(n,k)$ for small $n$ and $k$}
    \label{tab:123_132}
\end{table}

\begin{theorem}
\label{T:123_213}
$$\mathrm{pf}_n(123,213)=C_{n+1}-C_n$$
\end{theorem}

\begin{proof}
Suppose $\pi \in \mathcal{S}_n(123,213)$.  Since $\pi$ avoids 213, every digit before 1 must be larger than every digit after 1, and since $\pi$ avoids 123 all digits after 1 appear in decreasing order.  In other words, any permutation avoiding both of these patterns is a skew sum of permutations of the form $1 \oplus J_k$ with $k \geq 0$.  We will call these $1 \oplus J_k$ subpermutations the \emph{intervals} of the permutation.  Notice that any interval of size 1 is involved in no ascents.  However, an interval of size 2 or more has one ascent.  The digits of such an ascent may appear in separate blocks, or they may be in the same block, with an empty block later in the parking function.

We may encode such parking functions with sequences of dots and legally-arranged parentheses in the following way:  Draw $n$ dots, representing the $n$ digits of the $\{123,213\}$-avoiding permutation in the order they appear.  Before the first dot of each interval, place a left parenthesis.  If each number in that interval appears in a separate block, place a right parenthesis after the first dot.  If the first two numbers in that interval appear in a single block, there is an empty block later in the parking function.  Place a right parenthesis between the two dots representing the numbers that are on either side of that empty block.

This representation uniquely encodes the parking functions in question because the number of left parentheses gives the number of intervals and distances between successive left parentheses tell the sizes of the intervals.  The dots and the left parentheses uniquely identify the underlying permutation for the parking function.  The right parentheses then uniquely identify the location of any empty blocks, and their corresponding left parentheses tell which values appear in blocks of size 2.

Now, we give a recurrence that counts the dot and parenthesis arrangements.  As base cases, there is 1 way to arrange 0 dots, and there is 1 way to arrange 1 dot with a pair of parentheses as described; namely, $\left(\cdot\right)$. More generally, let $f(n)$ be the number of such arrangements with $n$ dots.  In the first case, we have only 1 dot inside the first pair of parentheses and in the second case we have 2 or more dots inside the first pair of parentheses.

In the first case, there may be $0 \leq j \leq n-1$ dots before the next left parenthesis (or the end of the arrangement).  This leaves $n-1-j$ remaining dots to arrange with parentheses and can be done in $f(n-1-j)$ ways.  Summing over all relevant values of $j$ gives $$\sum_{j=0}^{n-1}f(n-1-j)$$ ways to begin with one point inside the first pair of parentheses.  This corresponds to parking functions where every digit in the first interval appears in its own block of the parking function.

In the second case, there are $i \geq 2$ dots inside the first pair of parentheses.  There may still be $j \geq 0$ dots after the first pair of parentheses and before the next non-nested left parentheses.  There are $f(n-i-j)$ ways to arrange dots and parentheses after these initial points.  While there must be at least $k \geq 2$ dots inside the first parentheses and before the next left parenthesis, the remaining $i-k$ dots may be arranged with nested parentheses.  In this case, after summing over all possible values of $i$, $j$, and $k$, we get
$$\sum_{i=2}^{n}\sum_{j=0}^{n-i}\sum_{k=2}^{i}f(i-k)f(n-i-j)$$ possible arrangements of dots and parentheses where there are at least 2 dots inside the first pair of parentheses. This corresponds to parking functions where the first interval has size 2 or more and the first two digits of the first interval appear in a block of size 2.

Together we have that $f(0)=f(1)=1$ and
$$f(n)=\sum_{i=2}^{n}\sum_{j=0}^{n-i}\sum_{k=2}^{i}f(i-k)f(n-i-j)+\sum_{j=0}^{n-1}f(n-1-j).$$
We wish to show that $$f(n)=\frac{3(2n)!}{(n + 2)!(n - 1)!} = C_{n+1}-C_n.$$  We will prove this with computer-assisted induction.

First, direct computation from the recurrence shows that $f(n)=C_{n+1}-C_n$ for $0 \leq n \leq 10$.  Now, suppose that $f(k)=C_{k+1}-C_k$ for $k<n$ and consider the sum $\sum_{j=0}^{n-1}f(n-1-j)$.  By the induction hypothesis, 
$$\sum_{j=0}^{n-1}f(n-1-j) = \sum_{j=0}^{n-1}\frac{3(2(n-1-j))!}{((n-1-j) + 2)!((n-1-j) - 1)!},$$ and by CAS simplification we get that 
$$\sum_{j=0}^{n-1}f(n-1-j)=\frac{(n + 1)(2n)!}{((n + 1)!)^2}.$$
Also, considering the nested sum and invoking the induction hypothesis gives
$$f(n)=\sum_{i=2}^{n}\sum_{j=0}^{n-i}\sum_{k=2}^{i}f(i-k)f(n-i-j)=\frac{(2n^3 + 4n^2 - 2n - 4)(2n)!}{((n + 2)!)^2}$$ after simplification by CAS.

Together this gives that 
$$f(n)=\frac{(2n^3 + 4n^2 - 2n - 4)(2n)!}{((n + 2)!)^2} + \frac{(n + 1)(2n)!}{((n + 1)!)^2}=\frac{3(2n)!}{(n + 2)!(n - 1)!},$$ which is exactly what we wanted to show.
\end{proof}

Notice that together Theorem \ref{T:123_213} and Theorem \ref{T:132_213_312} provide an example of two different pattern sets of different sizes that yield the same enumerative result.  In particular, the proof of Theorem \ref{T:132_213_312} is bijective, while the proof of Theorem \ref{T:123_213} requires computer-assisted analysis.  It remains open to find a simpler bijective proof of Theorem \ref{T:123_213}.

\begin{theorem}
\label{T:132_231}
$$\mathrm{pf}_n(132,231)=\sum_{k=0}^{n-1}\binom{n-1}{k}C_{n-k}.$$
\end{theorem}

\begin{proof}
Suppose $\pi \in \mathcal{S}_n(132,231)$.
Since $\pi$ avoids 231, all digits before the 1 are decreasing, and since $\pi$ avoids 132, all digits after the 1 are increasing.  Therefore, the permutation is uniquely determined by choosing the digits that appear before the 1.  If 1 is in position $k+1$, this can be done in $\binom{n-1}{k}$ ways.

Further, considering the parking functions that correspond to such a permutation with 1 in position $k+1$, notice that the digits of the initial decreasing portion of the permutation must all appear in blocks of size 1.  However, 1 and all digits after it are in increasing order and may be written on any of the Dyck paths of size $n-k$.  Therefore, there are $C_{n-k}$ parking functions that correspond to any permutation that avoids 132 and 231 and has 1 in position $k+1$.  Summing over all possible values of $k$ gives the result.

\end{proof}

\begin{theorem}
\label{T:132_213}
$$\mathrm{pf}_n(132,213)=\mathrm{pf}_n(132,312)=\mathrm{pf}_n(213,231)=\mathrm{pf}_n(231,312).$$

Futhermore, $$\mathrm{pf}_n(132,213)=\mathrm{pf}_{n-1}(132,213)+2\sum_{i=1}^{n-1}\mathrm{pf}_i(132,213)\mathrm{pf}_{n-i-1}(132,213).$$
\end{theorem}

\begin{proof}
Suppose $\pi \in \mathcal{S}_n(132, 213)$.  Since $\pi$ avoids 132, all digits before $n$ are larger than all digits after $n$, and since $\pi$ avoids 213, all digits before $n$ appear in increasing order.  In other words, $\pi = I_a \ominus \pi^{\prime}$ where $\pi^{\prime} \in \mathcal{S}_{n-a}(132, 213)$.  This implies that $\pi$ is a skew-sum of increasing permutations and is uniquely described by the locations of its ascents.  Since each pair of adjacent digits in the permutation either forms an ascent or a descent, there are $2^{n-1}$ such permutations.  The same is true for the other three pattern pairs -- permutations avoiding them are uniquely described in terms of the location of their ascents.

Now, we consider how many members of $\mathcal{S}_n(132,213)$ can be drawn on a particular Dyck path.  While the labels along a particular vertical line segment must appear in increasing order, we have two choices between consecutive vertical segments: the last label on the first segment either starts an ascent or a descent.  In other words, the number of permutations drawn on a particular path corresponds to the number of two-colorings of the (EN) corners, and summing over all Dyck paths with two-colored (EN) corners gives the total number of pattern-avoiding parking functions.  Let $a_n$ be the number of such two colored paths.  Then we have $a_0=a_1=1$ and more generally, if the first return to $y=x$ is at $(n,n)$ there are $a_{n-1}$ options, while if the path returns to $y=x$ at $(i,i)$, then there are $a_i$ ways to write the initial path segment, $a_{n-i-1}$ ways to write the ending path segment and 2 choices for the color of the (EN) corner at $(i,i)$.  Summing over all possible values for $i$ gives the recurrence.
\end{proof}

The recurrence and initial conditions in Theorem \ref{T:132_213} match the recurrence for OEIS A001003, which is known as the Super-Catalan numbers.

Before we prove our next two results, we consider the following lemma about unlabeled Dyck paths with specific restrictions.

\begin{lemma}
\label{L:specfic_descent}
Let $h_{n,m}$ be the number of Dyck paths of semilength $n$ where the $m$th north step is immediately followed by an east step and let $C_n=\dfrac{\binom{2n}{n}}{n+1}$ be the $n$th Catalan number.  Then

$$h_{n,m}=\begin{cases}
0&m<1 \text{ or } m>n,\\
C_{n-1}&m=1,\\
C_{n}&m=n,\\
\displaystyle{\sum_{i=1}^{m-1}C_{i-1}h_{n-i,m-i}+\sum_{i=m}^nh_{i-1,m-1}C_{n-i}}&\text{ otherwise.}
\end{cases}.$$

\end{lemma}

\begin{proof}
The initial cases of this recurrence are easy to verify.  By definition, $1 \leq m \leq n$.  If $m=1$, then any Dyck path counted by $h_{n,m}$ begins with NE, and it may be followed by any of the $C_{n-1}$ Dyck paths of semilength $n-1$.  On the other hand if $m=n$, we know that the last north step in any Dyck path must be followed by an east step, and so $h_{n,m}$ counts all Dyck paths of semilength $n$.

For the general case where $1<m<n$, we decompose a Dyck path counted by $h_{n,m}$ according to its first return to the line $y=x$.  Let this first return be at the point $(i,i)$ where $1 \leq i \leq n$.  We consider two cases:  $i<m$ and $i \geq m$.

In the first case, we may use any of the $C_{i-1}$ Catalan paths between the first N and the first E that returns to $y=x$.  Now, we need the $(m-i)$th north step in the remainder of the path to be followed by an east step, so we may fill in the path after the first return in one of $h_{n-i,m-i}$ ways.

In the second case, any of the $C_{n-i}$ Catalan paths of semilength $n$ may appear after the first return to the line $y=x$.  For the initial part of the path, we have a north step, followed by a path of semilength $i-1$, followed by an east step.  If the $m$th north step of the entire path is followed by an east step, then the $(m-1)$st north step of the path between the initial N and the first E to return to $y=x$ must be followed by an east step.  There are $h_{i-1,m-1}$ ways to complete this portion of the path.

Summing over all relevant values of $i$ in each case gives the recurrence. 

\end{proof}

Although the recurrence of Lemma \ref{L:specfic_descent} gives an efficient way to compute $h_{n,m}$ for any value of $n$ and $m$, it remains open to find a non-recursive formula for it.  Interestingly, we conjecture that the triangle of numbers $h_{n,m}$ matches the values in entry A028364 of the OEIS.  It also remains an open problem to show that $h_{n,m}$ agrees with this triangle for arbitrarily large $n$ and $m$.  We are now ready for our next two enumerative results.  Each produces a sequence new to the OEIS.

\begin{theorem}
\label{T:132_321}
$$\mathrm{pf}_n(132,321)=C_n+\sum_{m=1}^{n-1}(n-m)\cdot h_{n,m}.$$
\end{theorem}

\begin{proof}
Suppose $\pi \in \mathcal{S}_n(132, 321)$.  These permutations are the reversals of the permutations in $\mathcal{S}_n(123, 231)$, which we analyzed in the proof of Theorem \ref{T:123_231}, and so by symmetry, either $\pi = I_n$, $\pi = I_m \ominus I_{n-m}$, or $\pi=(I_m \ominus I_\ell) \oplus I_{n-m-\ell}$.  Further, if we allow $\ell=n-m$, the second case is a special case of the third case, so we will consider them together.

There are $C_n$ parking functions corresponding to the increasing permutation $I_n$.  Now, for parking functions whose associated permutations are of the form $(I_m \ominus I_\ell) \oplus I_{n-m-\ell}$ where $1 \leq \ell \leq n-m$, notice that there is a descent immediately after the $m$th digit, and it is the only descent.   There are exactly $h_{n,m}$ Dyck paths where the $m$th north step is immediately followed by an east step, in order to accommodate this descent.  After the descent, we have $n-m$ choices for the remaining digits -- all of which are in increasing order, based on how many are smaller than the initial $m$ digits versus how many are larger.  Summing over all possible values of $m$ gives the result.
\end{proof}

\begin{theorem}
\label{T:213_321}
$$\mathrm{pf}_n(213,321)=C_n+\sum_{m=1}^{n-1}m\cdot h_{n,m}.$$
\end{theorem}

\begin{proof}
Suppose $\pi \in \mathcal{S}_n(213, 321)$.  These permutations are the reversals of the permutations in $\mathcal{S}_n(123, 312)$, which we analyzed in the proof of Theorem \ref{T:123_312}, and so by symmetry,
either $\pi=I_n$, $\pi=I_m \ominus I_{n-m}$, or $\pi=I_{\ell} \oplus (I_{m-\ell} \ominus I_{n-m})$.  Further, if we allow $\ell=0$, the second case is a special case of the third case, so we will consider them together.

There are $C_n$ parking functions corresponding to the increasing permutation $I_n$.  Now, for parking functions whose permutations are of the form $I_{\ell} \oplus (I_{m-\ell} \ominus I_{n-m})$ where $0 \leq \ell \leq m-1$, notice that there is a descent immediately after the $m$th digit, and it is the only descent.   There are exactly $h_{n,m}$ Dyck paths where the $m$th north step is immediately followed by an east step, in order to accommodate this descent.  Before the descent, we have $m$ choices for the remaining digits -- all of which are in increasing order, based on how many are smaller than the final $n-m$ digits versus how many are larger.  Summing over all possible values of $m$ gives the result.
\end{proof}

\begin{theorem}
\label{T:213_312}
$$\mathrm{pf}_n(213,312)=\sum_{k=0}^{n-1}\sum_{i=0}^{k} \binom{n-1}{i}\cdot \dfrac{(k + 1)\binom{2n - 2 - k}{n - 1 - k}}{n}.$$
\end{theorem}

\begin{proof}
Suppose $\pi \in \mathcal{S}_n(213,312)$.  Then since $\pi$ avoids 213, all digits before $n$ are increasing, and since $\pi$ avoids 312, all digits after $n$ are decreasing.  Since there are $\binom{n-1}{k}$ ways to choose the $k$ digits after $n$, there are $\binom{n-1}{0}+\binom{n-1}{1}+\cdots+\binom{n-1}{n-1}=2^{n-1}$ such permutations.  When we write such a permutation, with $k$ digits after $n$, on a Dyck path, the fact that these digits are in decreasing order forces them to be on single north steps, both preceeded and followed by east steps.

Now consider a Dyck path where the last $k$ north steps are single north steps.  Any permutation with at most $k$ digits after $n$ may be written on this Dyck path, so it is helpful to determine the number $d(n,k)$ of Dyck paths of size $n$ whose final $k$ north steps are single north steps.  We omit counting the initial north step in the path $(NE)^n$ since it is a single step by default.  We have $d(1,0)=d(2,0)=d(2,1)=1$ corresponding to NE, NNEE, and NENE respectively.  Now, consider a Dyck path of the form $(D_1)N(D_2)E$.  If $D_2$ is the empty path, then the number of single north steps at the end of the path increases by 1.  If $D_2$ is non-empty, since the first north step of $D_2$ is preceded by a north step, the single north steps at the end of $D_1$ are irrelevant, and the number of single north steps at the end of $(D_1)N(D_2)E$ is just equal to the single north steps at the end of $D_2$.

We have $$d(n,k)=d(n-1,k-1)+\sum_{i=k+1}^{n-1}C_{n-i-1}d(i,k)$$
where the first term accounts for paths ending in NE and the sum accounts for paths where $D_2$ is of semilength $i$ while $C_n$ denotes the $n$th Catalan number and there are $C_{n-i-1}$ ways to fill in the corresponding path $D_1$.  Computationally, the triangle $d(n,k)$ produces equal values to A033184 for small inputs; however, this particular recurrence for the values appears to be new.  Instead we will show that $d(n,k)$ satisfies the recurrence $d(n,k)=d(n,k-1)-d(n-1,k-2)$ with  boundary conditions that $d(n,0)=C_{n-1}$ and $d(n,n-1)=1$, which is the existing description of A033184.  Also $d(n,k)=0$ if $k<0$ since it is not possible to have a negative number of north steps.

For the boundary conditions, we notice that the unique way to get $n-1$ singleton north steps at the end of a path is to use the path $(NE)^n$.  On the other hand, for a path to have no singleton north steps at the end, the final north step must be the middle term in a NNEE factor.  Replacing this with NE produces a path of length $n-1$ with no restriction on the number of singleton north steps, so there are $C_{n-1}$ such paths.

We may rearrange the recurrence $d(n,k)=d(n,k-1)-d(n-1,k-2)$ as: $$d(n,k)+d(n-1,k-2)=d(n,k-1).$$
We will give a bijection between the paths accounted for on each side of this equation.

Suppose we have a path of semilength $n$ ending with $k-1$ singleton northsteps.  Either the path ends with ENE or it ends with EE.

If the path ends with ENE, then delete the final NE to get a path of semilength $n-1$ with $k-2$ singleton north steps at the end.  This is easily reversible by replacing the NE at the end.

If the path ends with EE, remove the final E.  Locate the last NN factor.  Find the longest consecutive string of Es followed by Ns that ends in this NN factor.  This subpath has the form $E^iN^j$ where $i>0$.  If $j \geq 3$, Replace this path with $E^iN^{j-1}EN$ to incorporate the final E and generate one more singleton north step.  If $j=2$, then this process actually generates more than one additional singleton north step.  To compensate for this, replace $E^iNN$ with $NE^{i+1}N$, which incorporates the final E and puts the first N at the end of a factor of other N steps so that we have only generated one additional singleton.

Since we have showed that lattice paths with $k$ singleton north steps at the end follow the same boundary conditions and recurrence as A033184, it follows that $d(n,k)=\dfrac{(k + 1)\binom{2n - 2 - k}{n - 1 - k}}{n}.$

Since $d(n,k)=\dfrac{(k + 1)\binom{2n - 2 - k}{n - 1 - k}}{n}$, then each of the paths counted by $d(n,k)$ can have $\sum_{i=0}^{k} \binom{n-1}{i}$ pattern avoiding permutations labeling its north steps, which gives the result.

\end{proof}

\begin{theorem}
\label{T:231_321}
$$\mathrm{pf}_n(231,321)=\dfrac{\binom{3n}{n}}{2n+1}.$$
\end{theorem}

\begin{proof}
It is known that $\dfrac{\binom{3n}{n}}{2n+1}$ counts the number of non-crossing trees on $n+1$ vertices (see \cite{DP93}).  Consider $n+1$ labeled points drawn on the perimeter of a circle.  Designate one to be the root vertex, and label the others from 1 to $n$.  A non-crossing tree is a rooted labeled tree that includes all $n+1$ vertices, and has no two edges that cross as chords of the circle.  We will show that there is a bijection between non-crossing trees on $n+1$ vertices and parking functions of size $n$ avoiding both 231 and 321.  

First, there is a simpler useful bijection between rooted ordered trees on $n+1$ vertices and unlabeled Dyck paths of semilength $n$.  Recursively, we see that both of these sets are counted by the Catalan numbers as follows:  For a rooted ordered tree, identify the rightmost edge $e$ emanating from the root.  Deleting $e$ decomposes the original tree into two subtrees: $T_1$ which is rooted at the original root of $T$ and $T_2$ which is rooted at the other endpoint of $e$.  Since $T_1$ and $T_2$ can be of any size where the number of vertices in $T_1$ and the number of vertices in $T_2$ adds to $n+1$, we see that rooted ordered trees satisfy the Catalan recurrence.  Similarly, we may decompose a Dyck path by identifying the east step that matches with the first north step.  This decomposes the Dyck path into $P=NP_1EP_2$ where $P_1$ and $P_2$ are Dyck paths whose semilengths sum to $n-1$.  For a recursive bijection, in the base case, the one tree on 2 vertices corresponds to the one Dyck path of semilength 1.  If we assume there is a bijection $f$ from rooted ordered trees on $i+1$ vertices to Dyck paths of semilength $i$ for $i<n$, then we have $f(T)=Nf(T_1)Ef(T_2)$, which captures the Catalan structure of both sets of objects.    The non-recursive interpretation of this bijection is as follows: Begin at the root, visit all edges $e_1, e_2,\dots e_k$ that have an endpoint at the root, then traverse the subtree below $e_1$, then the subtree below $e_2$, etc.  Rather than focusing on vertices, focus on the edges.  The first time an edge is traversed, record a north step.  Each time we start to visit the subtree below an edge, record an east step.  Since each edge is visited exactly twice (once along with all its sibling edges, and once as leading to its subtree), we get an equal number of north and east steps, and each east step is preceeded by the north step for that edge.

Now, we will show that all ways to arrange a rooted ordered tree on a circle as a non-crossing tree correspond to parking functions written on the corresponding Dyck path to that tree.

Notice that if $\pi \in \mathcal{S}_n(231,321)$, then either $\pi_1=1$ or $\pi_2=1$.  This is true recursively as well: once we know the location of the 1, the 2 must be in one of the earliest remaining two positions, and so on.  This is equivalent to the permutation being a direct sum of sub-permutations of the form $(1 \ominus I_k)$ where $k \geq 0$.  

Next, we give a process to fill in the entries of a parking function so that it gives a $\{231, 321\}$-avoiding permutation.  If we know the sizes of the blocks in a parking function, we can fill in the values of the blocks from left to right as follows.  If the first block has size $j$, we know the first $j-1$ entries are $1,2,\dots j-1$, however the last element may be $j$ or it may be larger.  Temporarily fill in this value with an $X$.  Now suppose the first $b$ non-empty blocks have been filled in with the digits $1,\dots, m$ plus the placeholder symbol $X$ and the next block to be filled has size $j$.  There are $j+1$ options for how this block may be filled.  One of the values $v$ of $\{m+1,\dots, m+j\}$ replaces the $X$ in the previous block and then the block consists of $\{m+1,\dots, m+j\} \setminus \{v\}$ followed by $X$, or $X$ remains in its current earlier block, and the new block has the entries $\{m+1,\dots, m+j\}$.  We repeat this process until all blocks have been considered, and we have a parking function filled with the entries $1,\dots, n-1$.  Replace the $X$ (which could be the last entry of any of the non-empty blocks) with $n$ to complete the parking function.

This algorithm for filling in the entries of a parking function can be translated to the placement of a rooted ordered tree as a non-crossing tree on a circle as follows:  First, denote the root vertex on the circle, which is contained in edges $e_1, e_2,\dots, e_k$.  Suppose $e_1, \dots, e_k$ are labeled so that their second endpoints go clockwise around the circle from the root.  Let $\left|T\right|$ denote the number of vertices in tree $T$.  Since we already know the sizes $\left|T_i\right|$ of the $k$ subtrees, we already know that $T_1$ occupies the first $\left|T_1\right|$ vertices counting clockwise from the root, $T_2$ occupies the next $\left|T_2\right|$ vertices, and so on.  The only decision that must be made is which of the first $\left|T_1\right|$ vertices is the second endpoint of $e_1$, which of the next $\left|T_2\right|$ vertices is the second endpoint of $e_2$, and so on.  In general, if $T_i$'s root has degree $d$, then there are $d+1$ options for the second endpoint of $e_i$: this is because $e_i$ could be left of all edges of $T_i$, right of all edges of $T_i$, or could be between two adjacent subtrees descending from $T_i$'s root.

We finally have all the ingredients to biject a non-crossing tree to a $\{231,321\}$-avoiding parking function.  First, consider the underlying rooted ordered tree, and find the corresponding Dyck path.  This Dyck path underlies our parking function and determines block sizes.  Each block of size $k$ corresponds to a collection of $k$ edges sharing the same root.  While the first block (corresponding to the edges descending from the root of the entire tree) must be $1,\dots,j-1$ plus the placeholder character $X$, for any subsequent block that must be populated with all but one of $\{m+1,\dots, m+j\} \cup \{X\}$, we look at the corresponding subtree in the non-crossing tree.  If the edge leading to the root of the subtree is to the right of all of the subtrees (reading their second vertices in clockwise order), leave $X$ in its current position and populate the block with $\{m+1,\dots, m+j\}$.  On the other hand, if the edge leading to the root of the subtree is to the left of $i$ subtrees, then replace $X$ with $m+i$ and populate the new block with $m+1,\dots, m+i-1,m+i+1,\dots, m+j, X$.

This correspondence is reversible.  Given a parking function, we can construct the underlying rooted ordered tree from the Dyck path. The values in a particular block give the endpoints of the edge leading to the relevant subtree.

\end{proof}

As an example of the bijection in the proof of Theorem \ref{T:231_321}, consider the tree in Figure \ref{F:tree_bij}, where the vertex labeled 0 represents the root.  The underlying tree itself is in bijection with the Dyck path $N^2ENEN^3E^2NE^3$, and therefore corresponds to a parking function with blocks of size 2, 1, 3, 0, 1, 0, and 0.

We now consider the orientation of the tree edges to fill the blocks one at a time as follows:

\begin{itemize}
    \item The first non-empty block has size 2, so it begins with entries $\{1,X\}$.
    \item So far, we have used the digit $m=1$.  Edge $a$ is the only edge descending from the root that has a nontrivial subtree below it.  Edge $a$ has its second vertex at vertex 2.  Reading clockwise, edge $a$ is to the left of one subtree, and therefore $m+1=2$ replaces $X$ in the first block.  The next non-empty block has size 1 and it is populated with $\{X\}$.  Our current partial parking function is $\{1,2\} \vert \{X\}$.
    \item So far we have used the digits 1 and $m=2$.  Edge $c$ is the only edge descending from vertex 2.  Reading clockwise, edge $c$ is to the left of 2 subtrees below vertex 4 (those beginning with edge $e$ and edge $f$).  So, $X$ is replaced by $m+2=4$ in the second block.  The next non-empty block has size 3 and it is populated with $\{3,5,X\}$.  Our current partial parking function is $\{1,2\} \vert \{4\}\vert \{3,5,X\}$.
     \item So far we have used the digits 1, 2, 3, 4 and $m=5$.  Edge $e$ is the only edge descending from vertex 4 with a nontrivial subtree below it.  Reading clockwise, edge $e$ is to the left of no subtrees below vertex 7.  So, $X$ remains in its current location.  The next non-empty block has size 1 and it is populated with $\{6\}$.  Our current partial parking function is $\{1,2\} \vert \{4\}\vert \{3,5,X\}\vert \emptyset \vert \{6\} \vert \emptyset\vert \emptyset$.
     \item We have found all but one digit of the parking function so we replace $X$ with $n=7$ to get the parking function $\{1,2\} \vert \{4\}\vert \{3,5,7\}\vert \emptyset \vert \{6\} \vert \emptyset\vert \emptyset$.

\end{itemize}

\begin{figure}[hbt]
    \centering

\begin{tikzpicture}

			\node[circle,draw,minimum size=0.2cm] (1) at (90:2cm) {0};
\node[circle,draw,minimum size=0.2cm] (2) at (45:2cm) {1};
\node[circle,draw,minimum size=0.2cm] (3) at (0:2cm) {2};
\node[circle,draw,minimum size=0.2cm] (4) at (315:2cm) {3};
\node[circle,draw,minimum size=0.2cm] (5) at (270:2cm) {4};
\node[circle,draw,minimum size=0.2cm] (6) at (225:2cm) {5};
\node[circle,draw,minimum size=0.2cm] (7) at (180:2cm) {6};
\node[circle,draw,minimum size=0.2cm] (8) at (135:2cm) {7};
\draw (1) -- node[right] {b}  ++(2);
\draw (1) -- node[left] {a}  ++(3);
\draw (3) -- node[right] {c}  ++(5);
\draw (5) -- node[below] {d}  ++(4);
\draw (5) -- node[below] {f}  ++(6);
\draw (5) -- node[right] {e}  ++(8);
\draw (8) -- node[right] {g}  ++(7);
    \end{tikzpicture}
    \caption{A non-crossing tree on 8 vertices}
    \label{F:tree_bij}
\end{figure}
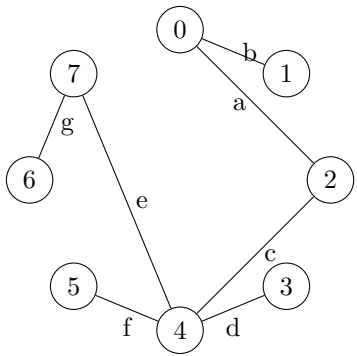

\begin{theorem}
\label{T:312_321}
Let $\mathcal{D}_n$ be the set of all Dyck paths of semilength $n$.  Given $d \in \mathcal{D}_n$, let $w(d)$ be the word of positive integers that consists of the lengths of the maximal consecutive strings of north steps of $d$.  

$$\mathrm{pf}_n(312,321)=\sum_{d \in \mathcal{D}_n} \prod_{i=1}^{\left|w(d)\right|-1} (w(d)_i+1).$$
\end{theorem}

While the statement of the formula in Theorem \ref{T:312_321} is more complicated than our prior results, the structure of these parking functions is remarkably similar to those counted in Theorem \ref{T:231_321}.   

\begin{proof}
Suppose $\pi \in \mathcal{S}_n(312,321)$.  Then either $\pi_{n-1}=n$ or $\pi_n=n$.  This is true recursively as well: once we know the location of the $n$, then $n-1$ must be in one of the last remaining two positions, and so on.    Next, we give a process to fill in the entries of a parking function so that it gives a $\{312, 321\}$-avoiding permutation.  If we know the sizes of the blocks in a parking function, we can fill in the values of the blocks from right to left as follows.  If the last block has size $j$, we know the last $j-1$ entries are $n-(j-2),n-(j-3),\dots n$, however the first element of this block may be $n-(j-1)$ or it may be smaller.  Temporarily fill in this value with an $X$.  Now suppose the last $b$ non-empty blocks have been filled in with the digits $n-m,\dots, n$ plus the placeholder symbol $X$ and the next block to be filled has size $j$.  There are $j+1$ options for how this block may be filled.  One of the values $v$ of $\{n-m-j,\dots, n-m-1\}$ replaces the $X$ in the previous block and then the block consists of $\{n-m-j,\dots, n-m-1\} \setminus \{v\}$ followed by $X$, or $X$ remains in its current earlier block, and the new block has the entries $\{n-m-j,\dots, n-m-1\}$.  We repeat this process until all blocks have been considered, and we have a parking function filled with the entries $2,\dots, n$.  Replace the $X$ (which could be the first entry of any of the non-empty blocks) with $1$ to complete the parking function.

While there was one choice for how the parking function was filled when we had only addressed the last block, each successive block of size $j$ has $j+1$ options for how it may be filled, so the total number of parking functions is the product of $w(d)_i+1$ where $w(d)_i$ is the size of the $i$th block and $i$ runs over all blocks except the final block.
\end{proof}

\section{Open Questions and Future Work}

In this paper we have surveyed the combinatorial sequences that arise when avoiding a set of patterns of length 3 in the context of parking functions. We have completely addressed the enumeration of parking functions avoiding any set of 2 or more such patterns.

As a direct result of our arguments, the following problems remain open:

\begin{enumerate}
    \item Show that the values $b(n+2,2)$ in the proof of Theorem \ref{T:123_132} match OEIS A000958 for all $n$.
    \item Find a proof for Theorem \ref{T:123_213} that is purely bijective, rather than requiring computer-assisted induction.
\item Find a non-recursive expression for the values $h_{n,m}$ in Lemma \ref{L:specfic_descent} and prove that these values match OEIS A028364.
\end{enumerate}

In addition, the majority of cases for avoiding a single pattern of length 3 remain open.  Remmel and Qiu's result for $\mathrm{pf}_n(123)$ was the motivation for this paper.  Brute force data for avoiding the other patterns of length 3 appears in Table \ref{tab:one_pattern}.  In general, these sequences appear to be much more challenging to describe than those considered earlier in this paper.  However, intriguingly, the initial data for avoiding 132 or avoiding 231 matches initial terms for OEIS entry A243688, which is described in that database as ``number of Sylvester classes of 1-multiparking functions of length $n$'', based on \cite{NT20}.

\begin{table}[hbt]
    \centering
    \begin{tabular}{|c|c|c|}
    \hline
    Pattern $P$&$\mathrm{pf}_n(P)$, $1 \leq n \leq 6$&OEIS\\
    \hline
\multirow{2}{*}{123}&\multirow{2}{*}{1, 3, 11, 48, 232, 1207}&new\\ &&(Remmel \& Qiu \cite{RQ17})\\
\hline
132&\multirow{2}{*}{1, 3, 13, 69, 417, 2759}&A243688\\
231&&(conjectured)\\
\hline
213&\multirow{2}{*}{1, 3, 14, 81, 533, 3822}&\multirow{2}{*}{new}\\
312&&\\
\hline
321&1, 3, 15, 97, 728, 6024&new \\
\hline
\end{tabular}
    \caption{\textit{Enumeration data for parking functions avoiding one pattern of length 3}}
    \label{tab:one_pattern}
\end{table}

This leads to one more explicit open problem, directly following our work:

\begin{enumerate}
\setcounter{enumi}{3}
    \item Find formulas for $\mathrm{pf}_n(\rho)$ where $\rho \in \mathcal{S}_3 \setminus \{123\}$.
    \end{enumerate}
    
More generally, we know that if there exists an ascent-preserving bijection between $\mathcal{S}_n(\rho_1, \dots, \rho_m)$ and     $\mathcal{S}_n(\tau_1, \dots, \tau_{\ell})$, then $$\mathrm{pf}_n(\rho_1, \dots, \rho_m)=\mathrm{pf}_n(\tau_1, \dots, \tau_{\ell}).$$  Such sets of patterns are said to be \emph{Wilf-equivalent}.  However, while this is a sufficient condition for Wilf equivalence, it is not a necessary condition, since, for example, Theorem \ref{T:132_213_312} and Theorem \ref{T:123_213} yield the same enumeration while the corresponding sets of pattern-avoiding permutations are not in bijection. This leads naturally to:

\begin{enumerate}
\setcounter{enumi}{4}
\item What other conditions imply two sets of patterns are Wilf-equivalent in the context of parking function patterns in this paper?
\end{enumerate}

More general open questions remain, such as avoiding longer patterns, or considering statistics on pattern-avoiding parking functions.  This paper is largely motivated by enumerative combinatorics, while parking functions tend to be algebraically motivated.  It would also be interesting to consider these definitions from an algebraic perspective.


\begin{thebibliography}{99}
\bibitem{B12} M. B\'{o}na, Combinatorics of Permutations (2nd ed.). Chapman and Hall/CRC, 2012.
\bibitem{CS08} G-S. Cheon and L.W. Shapiro, Protected points in ordered trees, \emph{Appl. Math. Letters} \textbf{21} (2008), 516--520.
\bibitem{DP93} S. Dulucq and J-G. Penaud,
Cordes, arbres et permutations, \emph{Discrete Math.} \textbf{117} (1993), 89--105.
\bibitem{GH96} A.M. Garsia and M. Haiman, A Remarkable q,t-Catalan Sequence and q-Lagrange Inversion, \emph{J. Algebraic Combin.} \textbf{5} (1996), 191--244.
\bibitem{H08} J. Haglund, The q,t-Catalan Numbers and the Space of Diagonal Harmonics, \emph{AMS University Lecture Series}, Volume: 41; (2008).
\bibitem{JM09} V. Jel\'{i}nek and T. Mansour, Wilf-equivalence on $k$-ary words, compositions, and parking functions, \emph{Electron. J. Combin.} \textbf{16} (2009), \#R58, 9pp. 
\bibitem{K11} S. Kitaev, Patterns in Permutations and Words. Springer, 2011.
\bibitem{KW66}  A. G. Konheim and B. Weiss, An occupancy discipline and applications, \emph{SIAM J. Applied Math.} \textbf{14} (1966), 1266--1274.
\bibitem{M48} T.S. Motzkin, Relations between hypersurface cross ratios and a combinatorial formula for partitions of a polygon, for permanent preponderance and for non-associative products, \emph{Bull. Amer. Math. Soc.} \textbf{54} (1948), 352--360.
\bibitem{NT20} J-C. Novelli and J-Y. Thibon,
Hopf algebras of $m$-permutations, $(m+1)$-ary trees, and $m$-parking functions,
\emph{Adv. in Appl. Math.} \textbf{117} (2020), Article 102019.
\bibitem{oeis} OEIS Foundation Inc. (2022), The On-Line Encyclopedia of Integer Sequences, http://oeis.org/.
\bibitem{P59} R. Pyke, The supremum and infimum of the Poisson process, \emph{Ann. Math. Statist.} \textbf{30} (1959), 568--576.
\bibitem{RQ17} J. Remmel and D. Qiu, Patterns in ordered set partitions and parking functions,  Permutation Patterns 2016 (slides), available electronically at \url{https://www.math.ucsd.edu/~duqiu/files/PP16.pdf}.
\bibitem{SS85} R. Simion and F. Schmidt, Restricted permutations, \emph{European J. Combin.} \textbf{6} (1985), 383--406.
\bibitem{Y15} C. H. Yan, Parking functions, \emph{Handbook of Enumerative Combinatorics}, Chapman and Hall/CRC, (2015), 835--894.

\end{thebibliography}
\end{document}